\documentclass[11pt, reqno]{amsart}
\usepackage{textcomp}
\usepackage{graphicx} 
\usepackage[utf8]{inputenc}
\usepackage{amsmath}
\usepackage{amsfonts}
\usepackage[a4paper, left=2.5cm, right=2.5cm, bottom=2.5cm,top=2.5cm, includefoot]{geometry}
\usepackage{xcolor}
\usepackage{tikz}
\usepackage{bbm}
\usepackage{pgfplots}
\pgfplotsset{width=10cm,compat=1.9}
\usepackage{mathtools}
\usepackage[most]{tcolorbox}
\usepackage{float}
\usepackage{subfig}
\usepackage{amsthm}
\usepackage{enumerate}
\usepackage{import}
\usepackage{dsfont}
\usepackage{amssymb}
\usepackage{physics}
\usepackage{tabu}
\usepackage{comment}
\pgfplotsset{compat=1.15}
\usepackage{mathrsfs}
\usetikzlibrary{arrows}
\usepackage[colorlinks=true,linkcolor=blue,citecolor=blue,urlcolor=blue]{hyperref}
\usepackage{cleveref}
\usetikzlibrary{intersections}
\usetikzlibrary{calc}
\usepackage{mathdots}
\usepackage{yhmath}
\usepackage{cancel}
\usepackage{color}
\usepackage{array}
\usepackage{multirow}
\usepackage{gensymb}
\usepackage{tabularx}
\usepackage{extarrows}
\usepackage{booktabs}
\usepackage{soul}
\usepackage[style=numeric,sorting=nyt,language=australian]{biblatex}
\usetikzlibrary{fadings}
\usetikzlibrary{patterns}
\usetikzlibrary{shadows.blur}
\usetikzlibrary{shapes}
\usepackage{ulem}
\usepackage{nicematrix}
\usepackage{microtype}
\usepackage{ragged2e}

\allowdisplaybreaks
\newtheorem{defn}{Definition}[section]
\newtheorem{thm}{Theorem}[section]
\newtheorem{cor}{Corollary}[section]
\newtheorem{lem}{Lemma}[section]
\newtheorem{prop}{Proposition}[section]

\theoremstyle{remark}
\newtheorem*{remark}{Remark}
\newtheorem*{notn}{Notation}
\newtheorem*{ackn}{Acknowledgements}

\Crefname{defn}{definition}{definitions}
\Crefname{defn}{Definition}{Definitions}

\Crefname{prop}{proposition}{propositions}
\Crefname{prop}{Proposition}{Propositions}

\Crefname{thm}{theorem}{theorems}
\Crefname{thm}{Theorem}{Theorems}

\Crefname{cor}{corollary}{corollaries}
\Crefname{cor}{Corollary}{Corollaries}

\Crefname{lem}{lemma}{lemmas}
\Crefname{lem}{Lemma}{Lemmas}

\numberwithin{equation}{section}

\title{An explicit power-saving bound for S\'{a}rk\"{o}zy's Theorem on shifted primes, assuming GRH}
\date{}
\author{Qijun (Davey) Gu, Minh Quang Nguy\^en}
\begin{document}
\begin{abstract}
    Assume the Generalised Riemann Hypothesis. If $A\subseteq \{1,...,N\}$ contains no two elements differing by $p+1$, $p$ prime, then $|A|\ll N^{\frac{89}{90} + o(1)}$.
\end{abstract}
\maketitle
\date{}

\pagestyle{plain} \setcounter{page}{1}

\maketitle

\tableofcontents

\section{Introduction}\label{sec1}

Given a subset $A\subseteq\{1,2,...,N\}$ with no elements differing by $p-1$, $p$ prime, we  ask how sparse $A$ is relative to $N$. In 1978, S\'ark\"{o}zy \cite{Sarkozy1978On} proved that $|A| = o(N)$, and in particular
$${|A|} \ll N(\log\log N)^{-2-o(1)}.$$ 
His work is built on Roth’s seminal work \cite{Roth1953On} on three-term arithmetic progressions in combination with the Hardy–Littlewood circle method.
Subsequently, further attempts were made to quantify this decay. In particular, Lucier \cite{Lucier2008Difference} used a density-increment argument to prove that
\[|A|\ll N(\log\log N)^{-\omega(N)}\]
for some function $\omega(N)\to\infty$ as $N\to\infty$. Ruzsa and Sanders \cite{Ruzsa2008Difference} then investigated the exceptional zeroes of the Dirichlet
$L$-functions in combination with the same density-increment strategy to obtain 
\[|A|\ll N\exp\qty({-C(\log N)^{1/4}}).\]
With more careful treatment of the error terms, Wang \cite{Wang2020On} improved the above exponent $1/4$ to $1/3$. A recent breakthrough came from Green in 2022 \cite{green2023sarkozystheoremshiftedprimes}, who established a power-saving bound
\begin{equation*}
    |A|\ll N^{1-c}
\end{equation*}
for some ineffective constant $c$. We prove an analogous theorem, but give a concrete power saving, at the cost of assuming the Generalised Riemann Hypothesis (GRH).

\begin{thm}\label{shifted primes}
Assume the Generalised Riemann Hypothesis. Given $A\subseteq \{1,2,...,N\}$, if the difference set $A - A$ does not contain $p+1$ for any primes $p$, then $$\abs{A} \ll N^{\frac{89}{90}+o(1)}.$$
\end{thm}

Our method works equally well for both the $p-1$ case and $p+1$ case, with little modification. This phenomenon is special to shifts by 1, as fairly trivial examples show that there is no such bound for the $p\pm a$ cases if $a\ne 1$. Indeed, for $\abs {a}>1$, pick $A = [N] \cap 4a\mathbb Z$; for $a = 0$, $A = [N]\cap 4\mathbb Z$ suffices.

Our proof begins with the observation that if the difference set $A - A$ contains no shifted primes $p + 1$, then for any arithmetic function $f$ supported on these shifted primes,
\begin{equation*}
    0 = \sum_{\substack{a_1-a_2+p+1 = 0: \\ a_i\in A,\ p\leq N - 1\text{ prime}}} f(p+1).
\end{equation*}
Crucially, this is in fact a triple additive convolution, which becomes a product in Fourier space:
\begin{equation}\label{eq:key observation with integral}
    0 = (\vb{1}_A * \vb{1}_{-A} * f)(0) = \int_0^1 |\widehat{\vb{1}_A}|^2(\alpha)\widehat{f}(\alpha)\dd\alpha.
\end{equation}
By periodicity, we can view this as an integral over the circle $S^1$. We then decompose it into two contributions, one near $\alpha = 0$ and one away from it. Heuristically, since we know that the integrand is $|A|^2\widehat{f}(0)$ at $\alpha =0$, we  expect the contribution near $\alpha =0$ to be real, positive and large. Thus, if we can then show that the contribution away from $\alpha = 0$ is not too negative, this will give us the desired asymptotic behaviour on $\abs{A}$.

Our goal for the rest of the paper is to construct an appropriate function $f$ and show that it yields the correct bounds, making the heuristic above rigorous. 

In an unpublished paper \cite{greenunpublished}, Green assumes GRH and shows that the power can be lowered to $11/12+o(1)$. Our approach achieves a weaker bound mainly because of \Cref{sieve and character twists}, which trivially disentangles the Fourier transform of a product. However, there is no other better bound in the literature. Further, while both Green's and our papers aim to prove that a certain Fourier transform is roughly nonnegative (see \Cref{minor arc bound on gw hat,major arc bound on gw hat,prop: f hat at 0}),  by translating this problem into one about additive convolutions, our proof offers an alternative,  more conceptually digestible approach. 
In particular, whereas Green makes substantial use of the Ramanujan sums and how they interact in a triple product, our approach is more inspired by the modern sieve methods: sieve weights arise naturally through multiplicative convolutions, hence the importance of \Cref{lem:convolution forms}. 

Besides the conceptual distinction, our paper uses only standard, accessible tools in analytic number theory, making the main ideas transparent. More generally, although the Fourier non-negative phenomenon is somewhat specific to this problem, the techniques developed in this paper are also useful for other problems in additive number theory involving primes and sieves (see, for example, the  more technical work of Grimmelt and Ter\"av\"ainen \cite{GrimmeltTeräväinen}).

The GRH, which is far from being settled, helps to show that the contribution from non-principal characters is negligible (see \Cref{lem 2} for more details). On the other hand, it is expected that sieve methods, such as the GPY sieve or the Selberg sieve, can give a good upper bound on this contribution, allowing us to proceed without GRH, at the cost of a worse bound. Similarly, Thorner and Zaman \cite{ThornerZaman} show that $|A|\ll N^{1-1/10^{18}}$ by using a refinement of Gallagher’s estimate. We plan to return to the unconditional problem in a future work.

We also note that L\^e and Spencer \cite{LeSpencer} study the analogous problem in function fields $\mathbb F_q[t]$, in which the GRH is valid. The result in their paper is improved by Fan and Lott \cite{FanLott} by adapting Green's strategy. 

\begin{ackn}
    This paper was the result of a project undertaken by both authors as part of the Summer Research in Mathematics programme, hosted by the Centre of Mathematical Sciences at the University of Cambridge. In particular, we would like to thank our supervisors Lasse Grimmelt for giving us the idea to consider the problem in a convolution setting, and Joni Ter\"{a}v\"{a}inen for double-checking our arguments. The second author was funded by the Trinity College Studentship Scheme.
\end{ackn}

\section{Overview of proof}\label{sec2}

In this section, we give a top-down overview of the proof strategy.

\begin{notn}
Throughout our paper, $N$ is universal and large, and we define associated length scales $R,Q,T$ to be global parameters. These will be powers of $N$, to be  determined later on.

We shall use asymptotic notation standard in analytic number theory, such as $\order{}$, $o()$, $\ll$. At various points, we also subscript this notation with parameters, so for example, $F(N)\ll_\theta G_\theta(N)$ means that for all $\theta$ (in the relevant range) there is some constant $C_\theta > 0$ such that, for all $N$, $F(N)\leq C_\theta G_\theta(N)$.

For any $M\in\mathbb N$, we use $[M]$ to denote the set $\{1,2,...,M\}$; for a subset $J\subseteq\mathbb Z$, we denote $-J :=\{-n:n\in J\}$.

We use $\ast$ and $\star$ to denote additive and multiplicative convolutions of two arithmetic functions, respectively; i.e., for $F,G:\mathbb Z\to \mathbb C $,
\begin{equation*}
    F\ast G(n):=\sum_{m_1 + m_2 = n} F(m_1)G(m_2),\qquad F\star G(n) :=\sum_{m_1m_2= n}F(m_1)G(m_2).
\end{equation*}
We use standard notation for the following arithmetic functions: $\Lambda$ for the von-Mangoldt function, $\phi$ for the Euler totient function, $\mu$ for the M\"{o}bius function. Define $\Lambda'$ as the modified von-Mangoldt function, where we have support on the primes. For $l,s\in\mathbb Z$, we write $(l,s):=\gcd(l,s)$ and $[l,s]:=\operatorname{lcm}(l,s)$. Given $n\in\mathbb N$, use $\nu_p(n)$ to denote the maximal exponent of the prime $p$ dividing $n$, and $\omega(n)$ to denote the number of its prime factors.

We denote by $\tau$ the divisor function, i.e. $\tau = 1\star 1$; for $\ell\geq 1$, denote by $\tau_\ell$ the $\ell$-fold divisor function.

A slightly unusual notation is that, if $l,s$ are coprime, $\overline{s}^{(l)}$ denotes the inverse of $s$ modulo $l$.

Additionally, all arithmetic functions in this paper (including the ones defined above) will have support cut off at $N$, so issues of convergence will not arise from summations. As a consequence, we follow the convention that (unless otherwise specified) all summation indices run up to $N$. Since this is not always reflected in the notation, we will remind the reader when relevant. Further, for $\ell \in\mathbb N$, we use the notation $(\mathbb Z/\ell\mathbb Z)^\ast = \{n\in\mathbb Z/\ell\mathbb Z: (n,\ell) = 1\}$, and sometimes we will use the shorthand $n = m(\ell)$ in lieu of $n\equiv m\pmod \ell$.

We use the notation $e(x):=\exp(2\pi i x)$; for $a,q\in \mathbb Z$, denote $e_q(a):=e\qty(\frac aq)$. Then for arithmetic functions $F:\mathbb Z\to \mathbb C$, take $\widehat{F}$ to be the Fourier transform, with convention
\begin{equation*}
    \widehat{F}(\alpha) := \sum_{n\in\mathbb Z}F(n)e(n\alpha).
\end{equation*}

For $\alpha\in \mathbb R$, we denote by $\norm{\alpha}_\mathbb Z$ the distance from $\alpha$ to the nearest integer. We use $\norm{\cdot}_1,\norm{\cdot}_2,\norm{\cdot}_\infty$ to denote the $L^1,L^2,L^\infty$ norms, in both physical and Fourier space; this should not be a source of confusion, since we already had notation distinguishing between physical and Fourier spaces.

The letters $X,\ell$ will be reserved for parameters local to lemmas and any other intermediary results, so that they can be used elsewhere if needed.
\end{notn}

We begin by making some key definitions which will be used globally throughout our paper.
\begin{defn}[Major and minor arcs]
Fixing $N$ and $R'\leq N$, we define the \emph{major arc} of the circle $S^1$ to be
\begin{equation*}
    \mathfrak M\qty(R'): = \bigcup_{1\leq q \leq R'}\bigcup_{a\in \mathbb Z/q\mathbb Z} \qty[\frac aq- \frac {R'}N,\frac aq + \frac {R'}N],
\end{equation*}
and subsequently define the \emph{minor arc} to be
\begin{equation*}
    \mathfrak m\qty(R') :=[0,1]\backslash \mathfrak M\qty(R').
\end{equation*}
\end{defn}

\begin{defn}[Ramanujan sums]\label{def:ramanujan sums}
For all $r\in\mathbb N$, $n\in\mathbb Z$, define
\begin{equation*}
    c_r(n):=\sum_{\substack{a\in(\mathbb Z/r\mathbb  Z)^\ast}}e_r(an).
\end{equation*}
\end{defn}

Inspired by Green's work \cite{green2023sarkozystheoremshiftedprimes}, we consider:

\begin{defn}[Functions for the proof]\label{def:functions for the proof}
Define the functions
\begin{align*}
    f(n) & :=\Lambda'(n-1)\tau_{4,T}(n)\Lambda_Q(n+1)w(n),\\ 
    g(n) & :=\Lambda_R(n-1)\tau_{4,T}(n)\Lambda_Q(n+1),
\end{align*}
where
\begin{align*}
    \tau_{4,T}(n)&:=\sum_{\substack{n=n_1n_2n_3n_4\\ T\le n_1,n_2,n_3\le 2T}}1,\\
    \Lambda_X(n)&:=\sum_{\ell\le X}\frac{\mu(\ell)}{\phi(\ell)}c_\ell(n),\\
    w(n)&:=\begin{cases}
        1 - \dfrac nN & \text{if } 1\le n\leq N,\\
        0 & \text{else}.
    \end{cases}
\end{align*}
\end{defn}
We note that $\Lambda_X$ was originally introduced by Heath-Brown \cite{Heath1985ternary}; it behaves similarly to $\Lambda$ on arithmetic progressions, but is easier to analyse. So, we wish to pass from $f$ to $gw$ first, via the following proposition.

\begin{prop}[Fourier-closeness]\label{fourier closeness lemma}
Assume GRH. Then for $ R\ll N^{1/5}$, have a constant $\delta_0>0$ such that
\begin{equation*}
    \|\widehat{f}-\widehat{gw}\|_\infty\ll N^{1+o(1)}\qty(R^{-1/2} + R^{4}/N^{1/2})QT^3\ll N^{1-\delta_0 + o(1)}.
\end{equation*}
\end{prop}

Next, we wish to bound $\widehat{gw}$.

\begin{prop}[Minor arc bound on $\widehat{gw}$]\label{minor arc bound on gw hat}
There exists a constant $\delta_1>0$ such that for all $\alpha\in\mathfrak m(Q)$,
\begin{align*}
    \qty|\widehat{gw}(\alpha)|&\ll N^{1+o(1)}Q^{-1}\ll N^{1-\delta_1 + o(1)}.
\end{align*}
\end{prop}

\begin{prop}[Major arc bound on $\widehat{gw}$]\label{major arc bound on gw hat}
Suppose $Q\ll R$. Then, for all $0<\theta<1$, there exists a constant $\delta_2>0$ such that, for all $\alpha\in\mathfrak M(Q)$,
\begin{equation*}
    \text{Re}\qty(\widehat{gw}(\alpha))\gg_{\theta} - \qty(NQ^{-\theta}+ NT^{-\theta} +  RQ^4T^3)N^{o(1)}\gg -N^{1-\delta_2 + o(1)}.
\end{equation*}
\end{prop}

\begin{remark}
    We will soon compute the constants $\delta_i$, but it might be helpful to consider them as functions of $R,Q,T$, which we vary to optimise our final power-saving constant.
\end{remark}

\begin{prop}\label{prop: f hat at 0} Suppose  $Q\ll R$, $RQ^4T^3 = o(N)$. Then there exists an absolute constant $\eta>0$, sufficiently small, such that for all $\alpha\in(-\eta/N,\eta/N)$, we have
\[\operatorname{Re}\qty(\widehat{gw}(\alpha))\gg N.\]
If, moreover, the same conditions as for \Cref{fourier closeness lemma} hold, then for all $\alpha\in(-\eta/N,\eta/N)$, we again have
    \begin{equation}\label{f hat at 0}
       \operatorname{Re}\qty( \widehat{f}(\alpha) )\gg N.
    \end{equation}
\end{prop}

We now show that these four results imply \Cref{shifted primes}.

\begin{proof}[Proof of \Cref{shifted primes}, assuming \Cref{fourier closeness lemma,minor arc bound on gw hat,major arc bound on gw hat,prop: f hat at 0}]

Suppose the difference set $A-A$ contains no shifted primes $p+1$. Then, recall our discussion in \Cref{sec1} and start from \eqref{eq:key observation with integral}. Define, for $0<\eta<1/4$ a fixed parameter found in \Cref{prop: f hat at 0},
\begin{equation}\label{key integral splits in two}
    0 = \underbrace{\int_{|\alpha|\leq \frac{\eta}{N}}}_{I_0} + \underbrace{\int_{\abs{\alpha} > \frac{\eta}{N}}}_{I_1}\abs{\widehat{\vb{1}_A}}^2(\alpha)\widehat{f}(\alpha)\dd\alpha.
\end{equation}
Focus first on $I_0$. By \Cref{prop: f hat at 0}, we know that $\text{Re}\qty(\widehat{f}(\alpha))> 0$. Then note that by \Cref{reality of integral}, $I_0\in\mathbb R$, so 
\begin{align*}
    I_0 & = \int\limits_{-\frac\eta N}^{\frac\eta N}|\widehat{\vb{1}_A}(\alpha)|^2\text{Re}\qty(\widehat{f}(\alpha))\dd \alpha\geq \int\limits_{-\frac\eta N}^{\frac\eta N}\text{Re}\qty(\widehat{\vb{1}_A}(\alpha))^2\text{Re}\qty(\widehat{f}(\alpha))\dd \alpha.
\end{align*}
For all $n\leq N$ and $\alpha \in \qty(-\eta /N, \eta/ N)$, we have $|2\pi n\alpha|\leq 2\pi\eta<\pi /2$. Thus, 
\begin{align*}
    \operatorname{Re}\qty(\widehat{\vb{1}_A}(\alpha)) = \sum_{n\in A}\cos(2n\pi \alpha)\ge |A|\cos(2\pi\eta)\gg_\eta |A|.
\end{align*}
Hence,
\begin{equation}\label{eq:localised contribution at 0}
    I_0\gg_\eta \abs{A}^2N\int_{\abs{\alpha}<\frac\eta N}\dd\alpha \gg \abs{A}^2.
\end{equation}

Next, for $I_1$, use \Cref{reality of integral} to get that $I_1\in \mathbb R$, so
\begin{align*}
    I_1 = \int_{\abs{\alpha}>\frac\eta N} |\widehat{\vb{1}_A}|^2(\alpha)\text{Re}\qty(\widehat{f}(\alpha))\dd\alpha.
\end{align*}
Then observe that for any $z\in \mathbb C$, $\text{Re}(z)\geq -\abs{\text{Re}(z)}\geq -\abs{z}$. Therefore, \Cref{minor arc bound on gw hat,major arc bound on gw hat} combine to give that for all $\alpha\in[0,1]$,
\begin{equation*}
    \text{Re}(\widehat{gw}(\alpha))\gg -N^{1-\min(\delta_1,\delta_2) + o(1)}.
\end{equation*}
Then using \Cref{fourier closeness lemma}, get that for all $\alpha \in [0,1]$,
\begin{align*}
    \text{Re}\qty(\widehat{f}(\alpha))&={\text{Re}(\widehat{f}(\alpha)-\widehat{gw}(\alpha))} + \text{Re}(\widehat{gw}(\alpha))\ge-\abs{\widehat{f}(\alpha)-\widehat{gw}(\alpha)}+\text{Re}(\widehat{gw}(\alpha))\\
    &\gg -N^{1-\delta_0 + o(1)}-N^{1-\min(\delta_1,\delta_2) + o(1)}\gg - N^{1-\delta + o(1)}
\end{align*}
where $\delta = \min(\delta_0,\delta_1,\delta_2)$.

Next, we use Plancherel's identity to obtain
\begin{align*}
     I_1&\gg -N^{1-\delta + o(1)}\int_{\abs{\alpha}>\frac\eta N}\abs{\widehat{\vb{1}_A}}^2(\alpha)\dd\alpha\\
     &\ge  -N^{1-\delta +o(1)}\int_0^1\abs{\widehat{\vb{1}_A}}^2(\alpha)\dd\alpha = -N^{1-\delta}\norm{\vb{1}_A}_2^2= -N^{1-\delta + o(1)} \abs{A}. 
\end{align*}

Recalling \eqref{key integral splits in two}, \eqref{eq:localised contribution at 0} and the above equation, we are done:
\begin{equation*}
|A|^2\ll I_0 = - I_1 \ll |A|N^{1-\delta + o(1)}.\qedhere
\end{equation*}
\end{proof}
This gives us \textit{some} power saving; to get the optimum power, we set $T = N^w, Q = N^x, R = N^y$, and substitute into the hypotheses and statements for \Cref{fourier closeness lemma,minor arc bound on gw hat,major arc bound on gw hat,prop: f hat at 0}. Then, we get $\delta$ as a function of $w,x,y$, which we want to maximise with constraints given by the hypotheses. Our problem becomes:
\begin{equation*}
    \begin{aligned}
        &\text{Maximise: }\delta = \min\qty(\theta x, \theta w, \frac 12y - x - 3w, \frac 12- x - 3w - 4 y,1-(4x + y + 3w) ) \\
        &\text{Subject to:}\quad  x\le y \le \frac15, \quad 4x + y + 3w < 1, \quad 0<\theta < 1.
    \end{aligned}
\end{equation*}
We leave this optimisation problem as an exercise to the reader; after some computation, we get 
$$T = Q =  N^{1/90},\qquad R= N^{ 1/9},\qquad \delta = \frac{\theta}{90}\to\frac{1}{90}$$
by taking $\theta \to 1$.

We dedicate the rest of the paper to proving \Cref{fourier closeness lemma,minor arc bound on gw hat,major arc bound on gw hat,prop: f hat at 0}. Below, we give a brief summary for each section to come.

In \Cref{sec3}, we prove \Cref{fourier closeness lemma}, using Fourier analysis. Here, we take $R$ as our major arc scale, and show that $\|\widehat{\Lambda}-\widehat{\Lambda_R}\|_\infty$ is small, by treating the major and minor arcs as separate cases. We then show that the three factors we attach in \Cref{def:functions for the proof} only contribute a power-saving loss. Notably, this is the only section of the proof  which uses GRH.

In \Cref{sec4}, we prove \Cref{minor arc bound on gw hat}, switching to $Q$ as our major and minor arc scale. We write $\widehat{gw}(\alpha)$ as a double sum, and, after switching the order of summation, see that the problem boils down to bounding $\norm{\lambda_X}_1$ and $\norm{\rho_T}_1$ (see \Cref{def:labmda and rho}).

The major arc bounds are the most technically challenging parts of the proof, so are split across the final two sections. In \Cref{sec5}, we give an overview of the proof for \Cref{major arc bound on gw hat,prop: f hat at 0}. This is the part of the proof which is the main novelty: the idea is to base ourselves on the major arc centred at $a/q$, and introduce a shorter length scale $L$. Then, we define the short interval sums $H_{q,a}(t)$, which are motivated by considering $\widehat{g}(a/q)$ as a sum, and truncating so that we only include $2L+1$ terms, centred on $t$. Crucially, we state \Cref{prop:decomp for H}, which says that  $H_{q,a}(t)$ is (up to error) a positive real constant. We then display that this result, combined with \Cref{fourier closeness lemma}, will be strong enough to imply \Cref{prop: f hat at 0} as a corollary. To conclude the section, we show how \Cref{prop:decomp for H} also yields the appropriate bound on $\text{Re}(\widehat{gw}(\alpha))$, for $\alpha$ in the entire $Q$-major arc around $a/q$, which gives \Cref{major arc bound on gw hat}.

Finally, we dedicate \Cref{sec6} to proving \Cref{prop:decomp for H}, and give an exact formula for the main term which arises. We use multiplicative properties of various exponential and arithmetic functions, as well as some careful bounding of error terms which arise. The main term turns out to be an Euler product, which we can bound with ease.

\section{Fourier-closeness}\label{sec3}
Since $gw$ approximates $f$ by swapping $\Lambda'$ for $\Lambda_R$, we shall first work towards upper-bounding $\|\widehat{\Lambda'}-\widehat{\Lambda_R}\|_\infty$; with the aid of a few subsequent lemmas, we can then pass to a bound on $\|\widehat{f}-\widehat{gw}\|_\infty$.

Let us first introduce the below definition.

\begin{defn}[Physical space major arc kernel $b_R$]
    Recalling \Cref{def:ramanujan sums}, we define the arithmetic function
    \begin{equation*}
    b_R(n):=\frac{R^2}{2N}\mathbf 1_{|n|\le N/R^2}\sum_{r\le R}c_r(n).
    \end{equation*}
\end{defn}

We proceed in stages, first bounding $\|\widehat{\Lambda\ast b_R} - \widehat{\Lambda_R}\|_\infty$. This is the only part of our proof which assumes GRH.
\begin{lem}\label{lem 2}
Assuming GRH, we have
\begin{equation*}
    \|\widehat{\Lambda*b_R}-\widehat{\Lambda_R}\|_\infty\ll N^{1+o(1)}\qty(R^{-1} + R^{4}N^{-1/2}).
\end{equation*}
\end{lem}
\begin{remark}
    Here, the $N$-dependency arises from the fact that we cut support off at $N$ for all arithmetic functions, as stated in \Cref{sec2}.
\end{remark}
\begin{proof}
    Consider 
\[\Lambda*b_R(n)=\sum_m\Lambda(n-m)b_R(m)=\frac{R^2}{2N}\sum_{r\le R}\sum_{\ell \in (\mathbb Z/r\mathbb Z)^\ast}\sum_{|m|\le N/R^2}\Lambda(n-m)e_r(\ell m).\] 
Rewrite the innermost sum using the substitution $n'=n-m$
\begin{align*}
    \sum_{|m|\le N/R^2}\Lambda(n-m)e_r(\ell m)&=e_r(\ell n)\sum_{n-N/R^2\le n'\le n+N/R^2}\Lambda(n')e_r(-\ell n')\\
    &=e_r(\ell n)\qty(\sum_{m\le X_1}\Lambda(m)e_r(-\ell m)-\sum_{m \le X_2}\Lambda(m)e_r(-\ell m)),
\end{align*}
where $X_1=n+N/R^2$ and $X_2=n-N/R^2-1.$ Now we let
\begin{equation*}
\Psi(X,r,n):=\sum_{\ell\in(\mathbb Z/r\mathbb Z)^*}e_r(\ell n)\sum_{m\le X}\Lambda(m)e_r(-\ell m),\qquad X\ge0.
\end{equation*} 
Then \Cref{lem: character sum} gives that
\begin{equation}
    \Psi(X,r,n)=\frac{\mu(r)}{\phi(r)}c_r(n)X+\order{r\sqrt{X}(\log(rX))^2}.\label{eq: Psi}
\end{equation}
For $1\le n\le N/R^2$, we have $X_2<0$ and $0\le X_1\le 2N/R^2$, so  \eqref{eq: Psi} gives
\[\sum_{\ell\in(\mathbb Z/ r\mathbb Z)^*} \sum_{|m|\le N/R^2}\Lambda(n-m)e_r(\ell m)=\Psi(X_1,r,n)=\frac{\mu(r)}{\phi(r)}c_r(n)X_1+\order{r\sqrt{\frac N{R^2}}(\log(rN))^2},\]
whence
\[\Lambda*b_R(n)=\frac{R^2}{2N}X_1\Lambda_R(n)+\order{\frac{R^4}{N}\sqrt{\frac{N}{R^2}}(\log (RN))^2}=\frac{R^2}{2N}X_1\Lambda_R(n)+\order{\frac{R^3}{\sqrt N}(\log (RN))^2}.\]
For $1\le n\le N/R^2$, we have $R^2 X_1/(2N)\in (1/2,1]$, so $|R^2 X_1/(2N)-1|\le 1/2$, so
\[|\Lambda*b_R(n)-\Lambda_R(n)|\le\frac{1}{2}|\Lambda_R(n)|+\order{\frac{R^3}{\sqrt N}(\log (RN))^2}.\]
Trivially, $|\Lambda_R(n)|\le R$, so for this range of $n$,
\[|\Lambda*b_R(n)-\Lambda_R(n)|=\order{R+\frac{R^3}{\sqrt N}(\log (RN))^2}.\]
Summing over this range gives
\[\sum_{n=1}^{N/R^2}|\Lambda*b_R(n)-\Lambda_R(n)|=\order{\frac{N}{R}+R\sqrt{N}(\log(RN))^2}.\]

For $n>N/R^2$, we have $X_1>X_2\ge0$, so we  use \eqref{eq: Psi}   to deduce 
\begin{align*}
   \sum_{\ell\in(\mathbb Z/ r\mathbb Z)^*} \sum_{|m|\le N/R^2}\Lambda(n-m)e_r(\ell m)&=\Psi(X_1,r,n)-\Psi(X_2,r_,n)\\
   &=\frac{\mu(r)}{\phi(r)}c_r(n)\,(X_1-X_2)+\order{r\sqrt N(\log(rN))^2}\\
   &=\frac{\mu(r)}{\phi(r)}\cdot\frac{2N}{R^2}c_r(n)+\order{r\sqrt N(\log(rN))^2}.
\end{align*}
Hence,
\begin{align*}
\Lambda*b_R(n)&=\frac{R^2}{2N}\sum_{r\le R}\left(\frac{\mu(r)}{\phi(r)}\cdot \frac{2N}{R^2}c_r( n)+\order{r\sqrt{N}\left(\log (RN)\right)^2}\right)\\
&=\Lambda_R(n)+\order{\frac{R^4}{\sqrt{N}}\left(\log (RN)\right)^2},
\end{align*}
whence
\[\sum_{n={N/R^2+1}}^N|\Lambda*b_R(n)-\Lambda_R(n)|=\order{\sqrt{N}R^4(\log(RN))^2}.\]
Thus,
\begin{align*}
    \|\widehat{\Lambda*b_R}-\widehat{\Lambda_R}\|_\infty & \le \|\Lambda* b_R-\Lambda_R\|_1 \\
    &=\order{ \frac{N}{R}+ \sqrt{N}R(\log(RN))^2 + \sqrt{N}R^4(\log(RN))^2}\\
    & = \order{N^{1+o(1)}\qty(R^{-1} + R^{4}N^{-1/2})}. \qedhere
\end{align*}
\end{proof}

The above result should motivate us to bound $\|{\widehat{\Lambda}-\widehat{\Lambda\ast b_R}}\|_\infty$. The following lemma is general, but we will only specialise it for $F = \Lambda$.

\begin{lem}[$b_R$ extracts the major arc contribution]\label{extract major arc}
Let $F$ be supported on $[N]$ and assume that $2\le R\ll N^{1/5}.$ We have
\[\|{\widehat{F}-\widehat{F* b_R}}\|_\infty\ll \sup_{\alpha\in\mathfrak m(R)}|\widehat{F}(\alpha)|+\norm{F}_1\qty(\frac{1}{R} + \frac{R^5}N).\]
\end{lem}
\begin{proof}
    Since convolution becomes multiplication on Fourier side, we have
    \[\sum_n(F*b_R)(n)e(\alpha n)=\widehat{F}(\alpha)\widehat{b_R}(\alpha).\]
    As $|{\widehat{F}(\alpha)}| \le \norm{F}_1$, it suffices to show that
    \begin{equation}
      \widehat{b_R}(\alpha)=\begin{cases}
        1+\order{1/R + R^5/N}&\text{if }\alpha\in\mathfrak M(R),\\ \order{1}&\text{if }\alpha\in\mathfrak m(R).
    \end{cases}  \label{estimate b_R}
    \end{equation}
    Given $\alpha\in\mathbb R,$ let $a\in\mathbb Z,q\in\mathbb N$ be such that $\|\alpha-a/q\|_{\mathbb Z}$ is minimal among all choices with $(a,q)=1,q\le R.$ Write $\beta=\alpha-a/q$ so that
    \begin{equation}
    \widehat{b_R}(\alpha)=\frac{R^2}{2N}\sum_{|n|\le N/R^2}e(-\beta n)+\sum_{r\le R}\sum_{\substack{1\le b<r\\ (b,r)=1\\ b/r\ne a/q}}\frac{R^2}{2N}\sum_{|n|\le N/R^2}e((\alpha-b/r)n).\label{hat b_R}
    \end{equation}
    We analyse these two terms in turn.
    
     \textbf{First term in \eqref{hat b_R}}: By the triangle inequality,
     \[\qty|\frac{R^2}{2N}\sum_{|n|\le N/R^2}e(-\beta n)|\ll 1,\]
     and the $\alpha\in\mathfrak m(R)$ case of \eqref{estimate b_R} follows.

     If $\alpha\in\mathfrak M(R)$, then $|\beta|\le R/N$ and we have
     \[\frac{R^2}{2N}\sum_{|n|\le N/R^2}e(-\beta n)=1+\order{\frac{R}{N}\cdot\frac{N}{R^2}},\]
     implying the $\alpha\in\mathfrak M(R)$ case of \eqref{estimate b_R}.

    \textbf{Second term in \eqref{hat b_R}}: Consider $b,r$ with $1\le b<r,(b,r)=1,b/r\ne a/q$. Let $b_0/r_0$ be the term that minimises $\norm{\alpha - b/r}_\mathbb Z$. Since
    \begin{equation*}
        \norm{\frac{b_0}{r_0}-\frac{a}{q}}_{\mathbb Z}=\norm{\frac{b_0q-r_0a}{qr_0}}_{\mathbb Z} \ge \frac{1}{R^2},
    \end{equation*}
    we have $\|\alpha-b_0/r_0\|_{\mathbb Z}\ge 1/\qty(2R^2)$ (since otherwise $a/q$ would not be closest to $\alpha$ among all fractions). For all other $1\le b<r,b\ne b_0$, we have
    \begin{align*}
        \norm{\alpha-\frac{b}{r}}_{\mathbb Z}=\norm{\alpha-\frac{b_0}{r}+\frac{b_0}{r}-\frac{b}{r}}_{\mathbb Z}\ge\frac{1}{2}\norm{\frac{b_0-b}{r}}_{\mathbb Z}.
    \end{align*}
    This shows that
    \begin{align*}
    \qty|\sum_{r\le R}\sum_{\substack{1\le b<r\\ (b,r)=1\\ b/r\ne a/q}}\frac{R^2}{2N}\sum_{|n|\le N/R^2}e((\alpha-b/r)n)|&\ll\sum_{r\le R}\frac{R^2}{2N}\qty(R^2+\sum_{1\le b\le r}\frac{r}{b})\\
    &\ll\sum_{r\le R}\frac{R^2}{2N}(R^2+r\log r)\ll \frac{R^5}{N}\ll 1,
    \end{align*}
    because $R\ll N^{1/5}.$
\end{proof}
This lemma motivates us to seek a minor arc bound for $|{\widehat{\Lambda}}|$.

\begin{prop}[Minor arc bound for $\Lambda$]\label{minor arc bound}
    Let $\alpha\in \mathfrak m(R)$. Then 
    \[|{\widehat{\Lambda}(\alpha)}|\ll N^{1+o(1)}\qty(R^{-1/2} + N^{-1/5}).\]
\end{prop}
\begin{remark}
    As stated in \Cref{sec2}, we use the convention that $\Lambda$ has its support cut off at $N$.
\end{remark}
\begin{proof}
    By Vinogradov's estimate (see Theorem 23.8 in \cite{Koukoulopoulos2019distribution}, say, for a proof), if we have $\alpha = a/q + \beta$, with $(a,q) = 1$ and $\abs{\beta}\le q^{-2}$, then
    \begin{equation*}
        |\widehat{\Lambda}|(\alpha)\ll (\log N)^{5/2}\qty(\frac{N}{\sqrt q} + \sqrt{qN} + N^{4/5}).
    \end{equation*}
    Then, by \Cref{rational approximations for minor arcs}, get that such $a,q,\beta$ exist, with $R<q \le N/R$, so
    \begin{equation*}
        \frac{N}{\sqrt{q}} + \sqrt{Nq} \ll R^{-1/2},
    \end{equation*}
    as needed.
\end{proof}
\begin{remark}
    This minor arc bound arises from Type I and Type II sum decompositions for the von Mangoldt function. The proof involves using Vaughan's identity to decompose the sum $\widehat{\Lambda}(\alpha)$ and then bounding each term.
\end{remark}

\begin{cor}\label{lem 1}
For $R\ll N^{1/5}$, we have
\begin{equation*}
    \|\widehat{\Lambda}-\widehat{\Lambda*b_R}\|_\infty\ll  N^{1+o(1)}\qty(1/R^{1/2} + R^5/N).
\end{equation*}
\end{cor}
\begin{proof}
    Use \Cref{extract major arc} (with $F = \Lambda$), \Cref{minor arc bound} and the crude bound $\norm{\Lambda}_1\le N\log N$. Then observing that $N^{-1/5} \ll R^{-1/2}$ since $R\ll N^{1/5}$, we are done.
\end{proof}

 Hence, our desired bound follows.
\begin{cor}\label{Lambda' and Lambda R are close}
Assuming GRH, for $R\ll N^{1/5}$,
\begin{equation*}
    \|\widehat{\Lambda'}-\widehat{\Lambda_R}\|_\infty\ll N^{1+o(1)}\qty(R^{-1/2} + R^{4}N^{-1/2}).
\end{equation*}
\end{cor}

\begin{proof}
    Apply the triangle inequality, followed by \Cref{lem 1} and \Cref{lem 2}, as well as \Cref{psi and theta bound} to get
    \begin{equation*}
        \begin{aligned}
            \|{\widehat{\Lambda'}-\widehat{\Lambda_R}}\|_\infty & \leq \|{\widehat{\Lambda}-\widehat{\Lambda*b_R}}\|_\infty +  \|{\widehat{\Lambda \ast b_R} - \widehat{\Lambda_R}}\|_\infty + \|{\widehat{\Lambda'}-\widehat{\Lambda}}\|_\infty \\
            & \ll N^{1+o(1)}\qty(R^{-1/2} + R^5/N + R^{-1} + R^{4}/N^{1/2}) + \sqrt N\\
            &\ll N^{1+o(1)}\qty(R^{-1/2} + R^{4}N^{-1/2}),
        \end{aligned}
    \end{equation*}
    since $\sqrt N\leq NR^{- 1/2}$ as $R\leq N$, and $R^5/N\ll R^{4}/N^{1/2}$ as $R\ll N^{1/3}$.
\end{proof}

We now wish to pass from this result to a bound on $\|{\widehat{f}-\widehat{gw}}\|_\infty$, so investigate how the factors $w,\tau_{4,T}$ and $\Lambda_Q$ contribute to the Fourier norm.

We first show that the Archimedean weight function $w$ only changes the Fourier norm up to a logarithmic factor.

\begin{lem}\label{archimedean weight changes L infty by log}
    Given any arithmetic function $F$ supported on $[N]$, we have
    \begin{equation*}
        \|{\widehat{Fw}}\|_\infty\ll \|{\widehat{F}}\|_\infty \log N .
    \end{equation*}
\end{lem}
\begin{proof}
    Using Fourier inversion, we have
    \begin{align*}
    \widehat{Fw}(\alpha)&=\sum_{n\le N}F(n)w(n)e(\alpha n) = \sum_{n\le N}F(n)e(\alpha n)\int_0^1\widehat w(x)e(-xn)\dd x\\
    &=\int_0^1\widehat{w}(x)\sum_{n\le N}F(n)e((\alpha-x) n)\dd x=\int_0^1\widehat{w}(x)\widehat{F}(\alpha-x)\dd x.
    \end{align*}
    Hence,
    \[\|\widehat{Fw}\|_\infty\le \| \widehat{F}\|_\infty \int_0^1|\widehat{w}(x)|\dd x.\]
    Here,
    \[\widehat{w}(x)=\sum_{n= 1}^N\qty(1-\frac{n}{N})e(xn),\]
    so by \Cref{A.4 in Green}, $|\widehat{w}(x)|\ll \min\{N,\|x\|_{\mathbb Z}^{-1}\}$. Thus,
    \begin{align*}
    \int_0^1|\widehat{w}(x)|\dd x\ll \int_0^1 \min\left\{N,\frac{1}{\|x\|_{\mathbb Z}}\right\}\dd x \ll \log N.&\qedhere
    \end{align*}
\end{proof}

The next lemma crudely estimates the effect of twisting by a sieve when calculating the Fourier norm.

\begin{lem}[Fourier property of sieve twist]\label{sieve and character twists}
Let $F: \mathbb N \to\mathbb C$ be of the
form $F(n) = 1 \star \lambda (n) = \sum_{d\mid n}\lambda(d)$, for some arithmetic function $\lambda$. Then for any arithmetic function $G$, we have
\begin{equation*}
    \|\widehat{FG}\|_\infty \le\|\lambda\|_1\|\widehat{G}\|_\infty.
\end{equation*}
\end{lem}
\begin{proof} We use the standard orthogonality relation $\sum\limits_{b\in(\mathbb Z/d\mathbb Z)}e\left(\frac{bn}{d}\right)=d\cdot\mathbf 1_{d\mid n}$ to get
    \begin{align*}
    \sup_{\alpha\in\mathbb R}\left|\sum_n F(n)G(n)e(\alpha n)\right|&= \sup_{\alpha\in\mathbb R}\left|\sum_d\lambda(d)\sum_{d\mid n}G(n)e(\alpha n)\right|\\
    &=\sup_{\alpha\in\mathbb R}\left|\sum_d\frac{\lambda(d)}{d}\sum_{b\in \mathbb Z/d\mathbb Z}\sum_n G(n)e\left(\left(\alpha+\frac{b}{d}\right)n\right)\right|&\\
    &\le \sum_d |\lambda(d)|\sup_{\alpha\in\mathbb R} \left|\sum_nG(n)e(\alpha n)\right|=\|\lambda\|_1\|\widehat{G}\|_\infty.&&\qedhere
    \end{align*}
\end{proof}
This lemma then motivates the following definitions.

\begin{defn}[$\lambda_Q$ and $\rho_T$]\label{def:labmda and rho}
Define the arithmetic functions $\lambda_Q$ and $\rho_T$ via
\begin{align*}
    \lambda_Q(d)&:=d\sum_{\substack{1\leq \ell \le Q\\\ell \equiv 0\pmod d}}\frac{\mu(\ell)}{\phi(\ell)}\mu\left(\frac{\ell}{d}\right)=d\mu(d)\sum_{\substack{1\leq \ell \le Q\\ \ell \equiv 0\pmod d}}\frac{\mu(\ell)^2}{\phi(\ell)},\\
    \rho_T(d) &:= \sum_{\substack{n_1n_2n_3 = d:\\ T\leq n_1,n_2,n_3\leq 2T}}1.
\end{align*}
\end{defn}
These are defined so that the following holds.
\begin{lem}[Convolution forms for $\Lambda_Q$ and $\tau_{4,T}$]\label{lem:convolution forms}
\begin{align*}
    \Lambda_Q = 1\star \lambda_Q,\qquad \tau_{4,T} = 1 \star \rho_T.
\end{align*}
\end{lem}
\begin{remark}
   This identity will be crucial in \Cref{sec4,sec5,sec6} later, as we will be working mostly with $\lambda_Q$ and $\rho_T$ instead of $\Lambda_Q$ and $\tau_{4,T}$. Also observe that $\lambda_Q$ is supported on $[Q]$ and $\rho_T$ is supported on $[8T^{3}]$.
\end{remark}
\begin{proof}
    For $\Lambda_Q$, use \Cref{alternative form for Ramanujan sums} to get
    \begin{align*}
        \Lambda_Q(n) & = \sum_{\ell\leq Q}\frac{\mu(\ell)}{\phi(\ell)}\sum_{d\mid (\ell,n)}d\mu\qty(\frac \ell d) = \sum_{\substack{d\leq Q\\ d\mid n}}\sum_{\substack{1\leq \ell \leq Q\\ \ell \equiv 0(d)}} \frac{\mu(\ell)d}{\phi(\ell )}\mu\qty(\frac \ell d)\\
        &=\sum_{d\mid n} d\sum_{\substack{1\leq \ell \leq Q\\ \ell \equiv 0 (d)}} \frac{\mu(\ell)}{\phi(\ell)}\mu\qty(\frac \ell d) = \sum_{d\mid n}\lambda_Q(d).
    \end{align*}
    For $\tau_{4,T}$, set $d =  n /{d_4}$ to get
    \begin{equation*}
        \tau_{4,T}(n) = \sum_{d\mid n} \sum_{\substack{T\leq n_1,n_2,n_3\leq 2T:\\ n_1n_2n_3 = d}}1=\sum_{d\mid n} \rho_T(d).\quad \qedhere
    \end{equation*}
\end{proof}

We conclude the section by proving \Cref{fourier closeness lemma}.
\begin{proof}[Proof of \Cref{fourier closeness lemma}]
Using \Cref{sieve and character twists} twice and then \Cref{archimedean weight changes L infty by log}, we have
    \begin{align*}
    \|\widehat{f}-\widehat{gw}\|_\infty=\|\widehat{(\Lambda-\Lambda_R)(\cdot)\tau_{4,T}(\cdot+1)\Lambda_Q(\cdot+2)}w(\cdot)\|_\infty\ll \|\widehat{\Lambda}-\widehat{\Lambda_R}\|_\infty\|\lambda_{Q}\|_1\|\rho_T\|_1\log N.
    \end{align*}
    Recalling \Cref{Lambda' and Lambda R are close}, it now remains to compute two $L^1$-norms.  
    \begin{equation}\label{L1 norm for rho}
    \norm{\rho_T}_1=\sum_d\sum_{\substack{n_1n_2n_3=d\\ T\le n_1,n_2,n_3\le 2T}}1 = T^{3},
    \end{equation}
    and
    \begin{align*}
    \|\lambda_Q\|_1&=\sum_dd\mu(d)^2\sum_{\substack{1\le \ell\le Q\\ q\equiv 0(d)}}\frac{\mu(\ell)^2}{\phi(\ell)}=\sum_{1\le \ell \le Q}\frac{\mu(\ell)^2}{\phi(\ell)}\sum_{d\mid \ell }d\mu(d)^2\\
    &=\sum_{1\le \ell\le Q}\frac{\mu(\ell)^2}{\phi(\ell)}\prod_{p\mid \ell}(1+p)=\sum_{1\le \ell\le Q}\mu(\ell)^2\prod_{p\mid \ell}\frac{p+1}{p-1}.
    \end{align*}

    Using Merten's second theorem, get a constant $c>0$ such that
    \begin{equation*}
        \prod_{p\mid \ell}\frac{p+1}{p-1} = \prod_{p\mid \ell}\qty(1 + \frac{2}{p-1})\leq \exp\qty(\sum_{p\mid \ell}\frac{2}{p-1})\leq \exp\qty(\sum_{p\leq \ell}\frac{4}{p})\le \exp(c\log \log \ell)\ll (\log Q)^{\order{1}}.
    \end{equation*}

    Crudely, $\sum_{q\leq Q}\mu(q)^2\le Q$, so we have that
    \begin{equation}\label{L1 norm for lambdaQ}
        \norm{\lambda_Q}_1 \ll Q(\log Q)^{\order{1}},
    \end{equation}
    and are done.
\end{proof}

\section{Minor arc bound for \texorpdfstring{$\widehat{gw}$}{gw} }\label{sec4}

Having defined $\rho_T,\lambda_Q$, we are ready to prove \Cref{minor arc bound on gw hat}.
\begin{proof}[Proof of \Cref{minor arc bound on gw hat}]
Suppose $\alpha\in \mathfrak m(Q)$. Then, using $\Lambda_Q = 1\star \lambda_Q$ and $\tau_{4,T} =1\star \rho_T$, we get
\begin{equation}
    \begin{aligned}
        \widehat{gw}(\alpha) &= \sum_{n\leq N}\Lambda_R(n-1)\Lambda_Q(n+1)\tau_{4,T}(n)w(n)e(n\alpha) \\
    &= \sum_{n\leq N}\sum_{d_1\mid n-1}\sum_{d_2\mid n+1}\sum_{d_3\mid n}\lambda_R(d_1)\lambda_Q(d_2)\rho_T(d_3)w(n)e(n\alpha)\\
    &= \sum_{\substack{d_1\leq R\\ d_2\leq Q\\T^{3}\le d_3\leq 8T^{3}}}\lambda_R(d_1)\lambda_Q(d_2)\rho_T(d_3)\sum_{\substack{n\leq N:\\ n \equiv 1 \text{ mod }d_1\\n\equiv -1 \text{mod }d_2\\ n\equiv 0 \text{ mod } d_3}}w(n)e(n\alpha) \label{minor arc section expression for gw hat}
    \end{aligned}
\end{equation}
First, observe that for the simultaneous congruences
\begin{equation}\label{prop 4.2 minor arc proof congruence}
\begin{cases}n\equiv 1 \mod d_1\\
    n  \equiv -1 \mod d_2\\
    n \equiv 0 \mod d_3\end{cases}
\end{equation}
to have a solution, we necessarily need $(d_1,d_3)=(d_2,d_3) = 1$ and $(d_1,d_2) \mid  2$. Hence, we make a definition.
\begin{defn}\label{def:mathcal G} Define
\begin{equation*}
    \mathcal G:=\left\{(d_1,d_2,d_3)\in [R]\times [Q]\times \qty[8T^{3}]: (d_1,d_3)= (d_2,d_3) = 1,\ (d_1,d_2)\mid 2\right\}
\end{equation*}
\end{defn}
so that all non-zero terms in the above sum arise from $\mathcal G$ (note our remark after \Cref{lem:convolution forms} about the support for $\lambda_X$ and $\rho_T$).
\begin{remark}
    The definition of $\mathcal G$ will be reused in all subsequent sections, and we will remind the reader when it next reappears. This is nothing more than notation to save writing out the coprimality conditions between the summation indices $d_j$.
\end{remark}

If $(d_1,d_2) = 1$, then by the Chinese remainder theorem, the above system \eqref{prop 4.2 minor arc proof congruence} has a unique solution modulo $d_1d_2d_3:=d$.

If $(d_1,d_2) = 2$, then writing $d_1 = 2d_1',d_2 = 2d_2'$ where $(d_1',d_2') = 1$, with no loss of generality $(2,d_1') = 1$, and so the above system is equivalent to
\begin{equation*}
    \begin{cases}
        n \equiv 1 \mod d_1'\\
        n\equiv -1 \mod d_2\\
        n\equiv 0 \mod d_3,
    \end{cases}
\end{equation*}
and again we may use the Chinese remainder theorem to get that this has one solution modulo $d_1'd_2d_3=  d/2$, and hence has two solutions mod $d$. Thus, letting $\mathcal S(d_1,d_2,d_3)$ be the set of solutions of \eqref{prop 4.2 minor arc proof congruence} modulo $d$, get that $\# \mathcal S(d_1,d_2,d_3) = 0,1,2$, non-zero if and only if $(d_1,d_2,d_3)\in\mathcal G$. Thus,
\begin{align*}
    \qty|\widehat{gw}(\alpha)| &\leq\sum_{\substack{d_1\leq R\\ d_2\leq Q\\d_3\leq 8T^3}}|\lambda_R(d_1)||\lambda_Q(d_2)|\rho_T(d_3)\qty|\sum_{\substack{n\leq N:\\ n\in \mathcal S(d_1,d_2,d_3)\pmod{d}}}w(n)e(n\alpha)|\\
    & \leq 2\sum_{(d_1,d_2,d_3)\in\mathcal G}|\lambda_R(d_1)||\lambda_Q(d_2)|\rho_T(d_3)\qty|\sum_{\substack{n\leq N:\\ n\equiv 0 \pmod{d}}}w(n)e(n\alpha)|\\
    &= 2 \sum_{(d_1,d_2,d_3)\in\mathcal G}|\lambda_R(d_1)||\lambda_Q(d_2)|\rho_T(d_3)\qty|\sum_{n\leq N/d}w(dn)e(dn\alpha)|\\
    & \ll \sum_{(d_1,d_2,d_3)\in\mathcal G}|\lambda_R(d_1)||\lambda_Q(d_2)|\rho_T(d_3)\min \qty(\frac{N}{d},\norm{d\alpha}^{-1}_\mathbb Z),
\end{align*}
by \Cref{A.4 in Green}. 

We now bound $\norm{\lambda_Q}_\infty$ by considering
\begin{equation*}
    |\lambda_Q(d)|\leq d\abs{\mu(d)}\sum_{\ell\leq Q/d}\frac{\mu(\ell d)^2}{\phi(\ell d)}= d\sum_{\ell \leq Q/d}\frac{\mu(\ell d)^2}{\phi(\ell d)}\le\frac{d}{\phi(d)}\sum_{\ell \le Q/d}\frac{\mu(\ell )^2}{\phi(\ell )}\ll \log \log d \log \qty(\frac Qd)
\end{equation*}
Here, we used the standard result that $\phi(d)\gg \frac{d}{\log \log d}$ and that $\sum_{n\leq X}\frac{\mu(n)^2}{\phi(n)}\sim \log X$ (\cite{Sitaramachandrarao1985On}, Lemma 2.3).

Define $\tau'_3(d) := \# \{(d_1,d_2,d_3)\in\mathcal G: d_1d_2d_3 = d\}$, which is supported on $d\leq 8RQT^3$. By the divisor bound, for all $\varepsilon>0$, $\rho_T(d)\leq \tau_3(d) \ll_\varepsilon d^\varepsilon$, so 
\begin{align*}
    \qty|\widehat{gw}(\alpha)|&\ll_\varepsilon \sum_{d\leq 8RQT^3}\tau_3'(d)(\log N)^4d^\varepsilon\min  \qty(\frac{N}{d},\norm{d\alpha}^{-1}_\mathbb Z).
\end{align*}
Then we use $\tau'_3(d)\leq \tau_3(d)\ll_\varepsilon d^\varepsilon$ to get
\begin{equation*}
    |\widehat{gw}(\alpha)|\ll_\varepsilon (8RQT^3)^{2\varepsilon}(\log N)^4\sum_{d\leq 8RQT^3}\min\qty(\frac{N}{d},\norm{d\alpha}^{-1}_\mathbb Z).
\end{equation*}
By \Cref{rational approximations for minor arcs}, since $\alpha\in\mathfrak m(Q)$, we can approximate $\alpha$ with $\alpha = \frac aq + \beta$, where $(a,q) = 1$, $Q<q\leq N/Q$ and $|\beta | < q^{-2}$. Then, using \Cref{sum over min}, we get
\begin{equation*}
    |\widehat{gw}(\alpha)|\ll_\varepsilon (RQT^3)^{2\varepsilon}(\log N)^4\qty(\frac Nq + q)\log (2RQT^3q)\ll_\varepsilon (RQT^3)^{2\varepsilon}(\log N)^5\frac N{Q}.
\end{equation*}
\end{proof}

\section{Major arc bound for \texorpdfstring{$\widehat{gw}$}{gw}}\label{sec5}
\begin{notn}
    For this entire section, fix $\alpha = \frac aq + \beta$, where $q\leq Q$ and $\abs{\beta}\leq Q/N$.
\end{notn}

Recall that in \Cref{major arc bound on gw hat}, we strive for a lower bound on $\text{Re}\qty(\widehat{gw}(\alpha))$; we already conducted an informal discussion of the proof strategy in \Cref{sec2}, which motivates our first definition.
\begin{defn}[Sums around short intervals]
    For $L =N/Q^2$ (which stays fixed for the rest of this section), define the arithmetic function
    \begin{equation*}
    H_{q,a}(t):=\frac{1}{2L+1}\sum_{|n-t|\leq L} g(n)e_q(an)
    \end{equation*}
\end{defn}
This next lemma will connect the above definition to $\widehat{gw}(\alpha)$.

\begin{lem}\label{gw hat as short interval sum plus error}
    Given $\alpha = a/q + \beta$, where $\abs{\beta}\leq Q/N$, we have
    \begin{equation*}
        \widehat{gw}(\alpha ) = \sum_{t}H_{q,a}(t)w(t)e(\beta t) + \order{Q^{-2}\norm{g}_1}.
    \end{equation*}
\end{lem}

\begin{proof}
    We start from the right-hand side, opening the definition of $H_{q,a}(t)$ and swapping sums:
    \begin{equation*}
        \sum_{t}H_{q,a}(t)w(t)e(\beta t) = \sum_{n\leq N}g(n)e_q(an) \sum_{|t-n|\leq L}\frac{1}{2L+1}e(\beta t) w(t).
    \end{equation*}
    So making the definitions
    \begin{equation*}
        \Phi(t):=w(t)e(\beta t),\quad \Phi_\text{avg}(n) := \frac{1}{2L+1}\sum_{|t-n|\le L}\Phi(t) = \frac{1}{2L+1}\sum_{\abs{h}\leq L} \Phi(n+h),
    \end{equation*}
    it now suffices to show that $\Phi_\text{avg}(n)= \Phi(n)+\order{\frac{1}{Q^2}}$, because $e_q(an)\Phi(n) = e(\alpha n)w(n)$.
    
    Since $w$ has a sharp cut-off at $N$, we define the modified smooth counterparts:
    \begin{equation*}
        \widetilde{w}(n):=1-\frac{n}{N},\quad \widetilde{\Phi}(t):=e(\beta t) \widetilde{w}(t),\quad \widetilde{\Phi}_\text{avg}(n):=\frac{1}{2L+1}\sum_{|t-n|\le L} \widetilde{\Phi}(t).
    \end{equation*}
    Let us first work with these: since $\widetilde{\Phi}$ is $\mathcal C^2$ (in fact holomorphic), we use Taylor's theorem to get
    \begin{equation*}
        \widetilde{\Phi}(t + h) = \widetilde{\Phi}(t) + h\widetilde{\Phi}'(t) + \frac12h^2\widetilde{\Phi}''(\xi)
    \end{equation*}
    for some $\xi = \xi(t,h) \in (t,t+h)$. Some computation reveals that
    \begin{equation*}
        \widetilde{\Phi}''(\xi) = e(\beta \xi) \qty(-(2\pi \beta)^2\qty(1-\frac \xi N) -\frac{4\pi i \beta}{N}).
    \end{equation*}
    So for $t\leq N+ L \ll N$, we have $\xi \ll N$, yielding $\abs{\widetilde{\Phi}''(\xi)}\ll \beta ^2 + \frac\beta N\ll \frac{Q^2}{N^2}$. Hence, for all $n\le N$,
    \begin{align*}
        \widetilde{\Phi}_\text{avg}(n) &= \frac{1}{2L+1}\sum_{\abs{h}\le L}\qty(\widetilde{\Phi}(n) + h\widetilde{\Phi}'(n) + \frac12h^2\widetilde{\Phi}''(\xi))\\
        & = \widetilde{\Phi}(n) + \order{\frac{Q^2}{N^2}L^2} = \widetilde{\Phi}(n)+ \order{\frac{1}{Q^2}}
    \end{align*}
    Note that $\Phi(t) = \widetilde{\Phi}(t)$ for all $t\leq N$, and so $\Phi_\text{avg}(n) = \widetilde{\Phi}_\text{avg}(n)$ for all $n\leq N-L$. Thence, we are left with the case in which $N-L<n\leq N$. Observe that
    \begin{align*}
        \abs{\Phi_\text{avg}(n) - \widetilde{\Phi}_\text{avg}(n)} \leq \frac{1}{2L+1} \sum_{\substack{t: |t-n|\le L,\\ t\ge N}}\abs{\widetilde{w}(t)}\leq \frac{1}{2L+1} \sum_{r = 1}^L\frac rN = \order{\frac{L}{N}} = \order{\frac{1}{Q^2}}.
    \end{align*}
    Thus, by the triangle inequality, we are done.
\end{proof}
This lemma is useful because we know (from \Cref{A.4 in Green}) that $\sum_t w(t)e(\beta t)\geq -\frac 12$, so if $H_{q,a}(t)$ splits into a positive real constant and error terms, we can proceed. The next lemma splits $H_{q,a}(t)$ into a $t$-independent term and an error.

\begin{lem}[Decomposition for $H_{q,a}(t)$]\label{decomp for H(qat)}
For any $q\in\mathbb N$, $a\in\mathbb Z$, we have
\begin{equation*}
    H_{q,a}(t) = H(q,a) + E(t),
\end{equation*}
with
\begin{equation}\label{eq:def of H(q,a)}
    H(q,a)  := \sum_{\substack{(d_1,d_2,d_3)\in\mathcal G:\\ q\mid [d_1,d_2,d_3]}}\frac{\lambda_R(d_1)\lambda_Q(d_2)\rho_T(d_3)}{[d_1,d_2,d_3]}e_q(an_0(d_1,d_2,d_3)),
\end{equation}
where $n_0(d_1,d_2,d_3)\in \mathbb Z/[d_1,d_2,d_3]\mathbb Z$ is uniquely defined by the solutions $\mathcal S(d_1,d_2,d_3)$.

Further,
\begin{equation*}
    \abs{E(t)}\ll q\frac{QRT^3}L N^{o(1)}.
\end{equation*}
\end{lem}

\begin{proof}
    We remind the reader of the notations $\mathcal G$ and $\mathcal S(d_1,d_2,d_3)$ from \Cref{def:mathcal G}, and write $(2L+1)H_{q,a}(t)$ as
    
\begin{align*}
    \sum_{|n-t|\le L}e_{q}(an)\sum_{\substack{d_1\le R,d_2\le Q,d_3\le 8T^3\\ d_1\mid n-1, d_2\mid n+1,d_3\mid n}}\lambda_R(d_1)\lambda_Q(d_2)&\rho_T(d_3)=\\
&\sum_{(d_1,d_2,d_3)\in\mathcal G}\lambda_R(d_1)\lambda_Q(d_2)\rho_T(d_3)\sum_{\substack{|n-t|\le L\\ n\in\mathcal S(d_1,d_2,d_3)}}e_q(an)
\end{align*}

Let $D=[d_1,d_2,d_3]$. By definition, $n\in\mathcal S(d_1,d_2,d_3)\Leftrightarrow n\equiv n_0(d_1,d_2,d_3) \pmod D$, so
\begin{equation*}
    (2L+1)H_{q,a}(t)=\sum_{(d_1,d_2,d_3)\in\mathcal G}\lambda_R(d_1)\lambda_Q(d_2)\rho_T(d_3)\sum_{\substack{|n-t|\le L\\ n\equiv n_0(d_1,d_2,d_3)\pmod D}}e_q(an).
\end{equation*}
Then we use:
\begin{equation*}
    \sum_{\substack{|n-t|\le L\\ n\equiv n_0\pmod{D}}}e_q(an)=\begin{cases}
    e_q(an_0)\qty(\frac{2L+1}{D}+\order{1})& \text{ if }q\mid D\\ \order{\min(\frac{L}{D}+1,\frac{q}{2})} & \text{ if } q\nmid D
\end{cases}
\end{equation*}
to get:
\begin{equation*}
    \begin{aligned}    (2L+1)H_{q,a}(t)=&(2L+1)\sum_{\substack{(d_1,d_2,d_3)\in\mathcal G:\\ q\mid D}}\frac{\lambda_R(d_1)\lambda_Q(d_2)\rho_T(d_3)}De_q(an_0(d_1,d_2,d_3))\\
    &+\order{\sum_{\substack{(d_1,d_2,d_3)\in\mathcal G\\ q\nmid D}}|\lambda_R(d_1)\lambda_Q(d_2)\rho_T(d_3)|\min\left\{\frac{2L}{D}+1,\frac{q}{2}\right\}}.\quad
    \end{aligned}
\end{equation*}
Now use $\min\{x+1,y\}\le y$ to get
    \begin{align*}
        \abs{E(t)}\ll \frac{1}{2L+1}\sum_{\substack{(d_1,d_2,d_3)\in\mathcal G}}|\lambda_R(d_1)\lambda_Q(d_2)\rho_T(d_3)|\qty(\frac q2)\ll \frac{q}{L}\norm{\lambda_R}_1\norm{\lambda_Q}_1\norm{\rho_T}_1.
    \end{align*}
    Recalling our calculations from \eqref{L1 norm for rho} and \eqref{L1 norm for lambdaQ}, we are done.
\end{proof}

We can further split the $t$-independent $H(q,a)$ into a main term and an error term:
\begin{prop}[Decomposition for $H(q,a)$]\label{prop:decomp for H}
Suppose $Q\ll R$. For all $q\in\mathbb N$, $a\in\mathbb Z$, we have
\begin{equation}\label{Decomp for H(q,a)}
    H(q,a) = H_0(q,a) + \mathcal E,
\end{equation}
where $q^{-1}\ll H_0(q,a)\ll q^{-1+o(1)}$, so in particular $H_0(q,a)\in\mathbb R_{>0}$. And, for all $0<\theta < 1$,
\begin{equation*}
   \abs{\mathcal E} \ll_{\theta} N^{o(1)}\qty( \frac{1}{Q^\theta} + \frac{1}{T^{\theta}}).
\end{equation*}
\end{prop}
This proof of this result is the most laborious of any in our paper, so we postpone it to \Cref{sec6}.

Further assuming \Cref{prop:decomp for H}, we can bring in \Cref{gw hat as short interval sum plus error,decomp for H(qat)} to get the full decomposition for $\widehat{gw}$:
\begin{equation}\label{full decomp for gw on major arc}
    \widehat{gw}\qty(\frac{a}{q}+\beta)=(H_0(q,a)+\mathcal E)\sum_t w(t)e(\beta t)+\sum_tE(t)w(t)e(\beta t)+\order{Q^{-2}\norm{g}_1}.
\end{equation}
It remains to control the error terms; we have already taken care of the first two, so move onto the third.
\begin{lem}\label{L1 norm of g}
    We have
    \begin{equation*}
        Q^{-2}\norm{g}_1\ll N^{1+o(1)}Q^{-2}.
    \end{equation*}
\end{lem}
\begin{proof}
    Observe that
    \[|\Lambda_Q(n)|\le \sum_{d\mid n}|\lambda_Q(d)|\ll_\varepsilon \tau(n) n^\varepsilon\ll n^{o(1)},\qquad \tau_{4,T}(n)\le \tau_4(n)\ll n^{o(1)},\]
    so $\norm{g}_\infty \ll N^{o(1)}$ and we're done.
\end{proof}

We are now ready to prove \Cref{major arc bound on gw hat} (we remind the reader that we were trying to lower-bound $\operatorname{Re}\qty(\widehat{gw}(\alpha))$).

\begin{proof}[Proof of \Cref{major arc bound on gw hat}, assuming \Cref{prop:decomp for H}]
Recalling the decomposition \eqref{full decomp for gw on major arc}, we bound the main term carefully, and then use the triangle inequality to bound all error terms in modulus.

First, $Q\ll R$, so by the main term in \Cref{prop:decomp for H} and \Cref{A.4 in Green}, we have
\begin{equation*}
    \text{Re}\qty(H_0(q,a)\sum_t w(t)e(\beta t))\ge -\frac 12 H_0(q,a)\gg-\frac{1}{q^{1-o(1)}}\gg -1
\end{equation*}
as $q\ge 1$.

Then, using the error bound in \Cref{prop:decomp for H}, we have, for all $0<\theta<1$,
\begin{equation*}
    \abs{\mathcal E\sum_t w(t)e(\beta t)}\ll_{\theta} N^{1+o(1)}(Q^{-\theta} + T^{-\theta})
\end{equation*}
Next, note that since $\alpha\in\mathfrak M(Q)$, we have $q\leq Q$, and use the final part from \Cref{decomp for H(qat)} to get
\begin{equation*}
    \qty|\sum_tE(t)w(t)e(\beta t)|\le N\max_t |E(t)|\ll N\frac{Q^2RT^3}L N^{o(1)} \ll RQ^4T^3N^{o(1)}.
\end{equation*}
To conclude, \Cref{L1 norm of g} takes care of the final error term, so we are done.
\end{proof}

As a direct consequence, this proof gives us \Cref{prop: f hat at 0} almost for free.

\begin{proof}[Proof of \Cref{prop: f hat at 0}, assuming \Cref{prop:decomp for H}] Around 0, we may take $a=0,q=1$ (so $\beta=\alpha$), and since $|\alpha|\le \eta/N$, we have by \Cref{A.4 in Green} that
\begin{equation*}
    \operatorname{Re}\qty(H_0(q,a)\sum_t w(t)e(\beta t))=H_0(1,0)\qty[\frac{1}{2N}\qty(\frac{\sin(\pi\alpha N)}{\sin(\pi\alpha)})^2-\frac{1}{2}]\gg N,
\end{equation*}
where we have used a small angle approximation in the inequality, recalling that $H_0(1,0)$ is a constant and $\eta$ is sufficiently small. Recalling \eqref{full decomp for gw on major arc}, the first part of \Cref{prop: f hat at 0} follows since the error terms are all $o(N)$.

Then we use \Cref{fourier closeness lemma} to get
\begin{equation*}
    \operatorname{Re}\qty(\widehat{f}(\alpha) )\geq \operatorname{Re}\qty(\widehat{gw}(\alpha)) - \norm{\widehat{gw}-\widehat{f}}_\infty \gg N - N^{1-\delta_0 + o(1)}\gg N. \qquad \qedhere
\end{equation*}
\end{proof}

\section{Short interval sums}\label{sec6}

We use this section to prove \Cref{prop:decomp for H}, and invite the reader to remind themselves of the definition of $H(q,a)$ in \eqref{eq:def of H(q,a)}. We anticipate that we will need to split cases as follows:
\begin{defn}
For $i = 1,2$, let
\begin{align*}
    \mathcal G_i&:=\{(d_1,d_2,d_3)\in\mathcal G: (d_1,d_2) = i\}, \\
    H^{(i)}(q,a) & := \sum_{\substack{(d_1,d_2,d_3)\in\mathcal G_i:\\ q\mid [d_1,d_2,d_3]}}\frac{\lambda_R(d_1)\lambda_Q(d_2)\rho_T(d_3)}{[d_1,d_2,d_3]}e_q(an_0(d_1,d_2,d_3)).
\end{align*}
\end{defn}
This is so that we get
\begin{equation*}
    H(q,a) = H^{(1)}(q,a) + H^{(2)}(q,a),
\end{equation*}
and we can analyse these two terms in turn. For the remainder of this section, most results pertaining to $H^{(1)}(q,a)$ will have an analogue for $H^{(2)}(q,a)$, the proofs of which are not too dissimilar.

Since this is the longest section of our paper, we give a brief overview: we wish to rewrite $H^{(i)}(q,a)$ as a sum of exponentials (\Cref{lem:H as a sum of exponentials}), whose coefficients we can analyse. We split the coefficients into a multiplicative part (\Cref{prop:main terms for S}), with two error terms to be bounded separately (\Cref{bound on E(1)} and \Cref{bound on E(2)}). The multiplicative part will then interact well with our exponential phases (\Cref{prop:multiplicative sums}), yielding the main constant term, whilst the error terms can be bounded in modulus, so we do not need to see how they interact with the exponentials.

\begin{lem}\label{lem:H as a sum of exponentials}
We have for all $q\in\mathbb N,\ a\in\mathbb Z$ that
\begin{equation}\label{eq:H1 exponential sum}
    \begin{aligned}
        H^{(1)}(q,a)=&\sum_{q_1q_2q_3=q}e_{q_1}\left(a\overline{q_2q_3}^{(q_1)}\right)e_{q_2}\left(-a\overline{q_1q_3}^{(q_2)}\right)\sum_{\substack{(d_1,d_2,d_3)\in\mathcal G_1:\\ q_i\mid d_i}}\frac{\lambda_R(d_1)\lambda_Q(d_2)\rho_T(d_3)}{d_1d_2d_3}.
    \end{aligned}
\end{equation}
Further, let $q$ be odd. Then,
    \begin{equation}\label{eq:H2 exponential sum q odd}
        \begin{aligned}
            H^{(2)}(q,a) & = \sum_{q_1q_2q_3 = q}e_{q_1}\qty(a\overline{q_2q_3}^{(q_1)})e_{q_2}\qty(-a\overline{q_1q_3}^{(q_2)})\sum_{\substack{(d_1,d_2,d_3)\in \mathcal G_2':\\ q_i\mid d_i}}\frac{\lambda_R(2d_1)\lambda_Q(2d_2)\rho_T(d_3)}{2d_1d_2d_3},
        \end{aligned}
    \end{equation}
    where
    \begin{equation*}
        \mathcal G_2':=\left\{(d_1,d_2,d_3)\in \qty[\frac{R}2]\times\qty[\frac{Q}{2}]\times \qty[8T^3]: (d_1,d_2) = (d_2,d_3) = (d_3,d_1) = (d_3,2) = 1\right\}.
    \end{equation*}
    Next, let $q = 2m$, $m$ odd. We then have that
    \begin{equation}\label{eq:H2 exponential sum q even}
        H^{(2)}(q,a) = e_2(a) \sum_{q_1q_2q_3 = m}e_{q_1}\qty(a\overline{2q_2q_3}^{(q_1)})e_{q_2}\qty(-a\overline{2q_1q_3}^{(q_2)})\sum_{\substack{(d_1,d_2,d_3)\in \mathcal G_2':\\ q_i\mid d_i}}\frac{\lambda_R(2d_1)\lambda_Q(2d_2)\rho_T(d_3)}{2d_1d_2d_3}.
    \end{equation}
    Finally, if $\nu_2(q) \ge 2$, then
    \begin{equation*}
        H^{(2)}(q,a) = 0.
    \end{equation*}
\end{lem}

\begin{proof}
For the first equation \eqref{eq:H1 exponential sum}, since $(d_1,d_2,d_3)\in\mathcal G_1$, $d_1,d_2,d_3$ are pairwise coprime, so $[d_1,d_2,d_3] = d_1d_2d_3$, and for $q\mid d_1d_2d_3$, we can write $q=q_1q_2q_3$ uniquely such that $q_i\mid d_i$, so the sum splits as
\begin{equation*}
    \begin{aligned}
        H^{(1)}(q,a) &= \sum_{\substack{(d_1,d_2,d_3)\in\mathcal G_1:\\ q\mid d_1d_2d_3}}\frac{\lambda_R(d_1)\lambda_Q(d_2)\rho_T(d_3)}{d_1d_2d_3}e_q(an_0(d_1,d_2,d_3))\\
        & =  \sum_{q_1q_2q_3 = q}\sum_{\substack{(d_1,d_2,d_3)\in\mathcal G_1:\\ q_i\mid d_i}}\frac{\lambda_R(d_1)\lambda_Q(d_2)\rho_T(d_3)}{d_1d_2d_3}e_q(an_0(d_1,d_2,d_3)).
    \end{aligned}
\end{equation*}

Recalling the definition of $n_0$ from \Cref{decomp for H(qat)}, we get $n_0\equiv 1\pmod {q_1},n_0\equiv -1\pmod{q_2},n_0\equiv 0\pmod{q_3}$, so then by the Chinese remainder theorem,
    \begin{equation*}
        n_0\equiv q_2q_3\overline{q_2q_3}^{(q_1)}-q_1q_3\overline{q_1q_3}^{(q_2)}\pmod {q},
    \end{equation*}
    so the phase becomes
    \begin{equation*}
        e_q(an_0)=e_q\left(aq_2q_3\overline{q_2q_3}^{(q_1)}\right)e_q\left(-aq_1q_3\overline{q_1q_3}^{(q_2)}\right)=e_{q_1}\left(a\overline{q_2q_3}^{(q_1)}\right)e_{q_2}\left(-a\overline{q_1q_3}^{(q_2)}\right),
    \end{equation*}
    as claimed.
    
    Moving onto the second equation \eqref{eq:H2 exponential sum q odd}, if $(d_1,d_2,d_3)\in\mathcal G_2$, we can set $d_1 = 2d_1'$ and $d_2 = 2d_2'$. Then, manifestly, $(d_1,d_2,d_3)\in\mathcal G_2$ if and only if $(d_1',d_2',d_3)\in\mathcal G_2'$. The system of congruences $n\equiv 1 \pmod{d_1},\ n\equiv -1\pmod{d_2},\ n\equiv 0\pmod{d_3}$ has a unique solution modulo $[d_1,d_2,d_3]=2d_1'd_2'd_3$. Moreover, by the oddity of $q$, $q\mid 2d_1'd_2'd_3\Rightarrow q\mid d_1'd_2'd_3$.
    
    The rest of the proof follows exactly as above (we invite the reader to think about why this is).

    For the third equation \eqref{eq:H2 exponential sum q even}, we note that all of the above paragraph holds, up until the final sentence. Now, $2m\mid2 d_1d_2d_3\Leftrightarrow m\mid d_1d_2d_3$, so we get
    \begin{equation*}
        H^{(2)}(2m,a)=\sum_{q_1q_2q_3 = m}\sum_{\substack{(d_1,d_2,d_3)\in\mathcal G_2':\\ q_i\mid d_i}} \frac{\lambda_R(2d_1)\lambda_Q(2d_2)\rho_T(d_3)}{2d_1d_2d_3}e_{2m}(an_0(d_1,d_2,d_3)),
    \end{equation*}
    where $n_0(d_1,d_2,d_3)$ is the solution to
    \begin{equation*}
        \begin{aligned}
            \begin{cases}
                n_0 \equiv 1 \pmod{2d_1}\\ n_0 \equiv -1 \pmod{2d_2}\\ n_0\equiv 0\pmod{d_3}
            \end{cases} \Rightarrow \begin{cases}
                n_0\equiv 1\pmod{q_1}\\ n_0\equiv -1\pmod{q_2}\\ n_0\equiv 0 \pmod{q_3}\\ n_0\equiv 1\pmod{2}.
            \end{cases}
        \end{aligned}
    \end{equation*}
    So, by the Chinese Remainder theorem, get
    \begin{equation*}
        n_0\equiv 2q_2q_3\overline{2q_2q_3}^{{(q_1)}} - 2q_1q_3\overline{2q_1q_3}^{(q_2)}+m\pmod{q},
    \end{equation*}
    and the rest of the proof follows exactly as before.

    For the final case, note that if $(d_1,d_2,d_3)\in\mathcal G_2$, then $\nu_2([d_1,d_2,d_3]) = 1$, because in particular, $d_3$ must be odd. Hence, $\nu_2(q) \geq 2 \Rightarrow q\nmid [d_1,d_2,d_3]$, so the sum for $H^{(2)}(q,a)$ is in fact empty.
\end{proof}
\begin{remark}
    Henceforth, whenever we deal with $H^{(2)}(q,a)$, we can always assume $\nu_2(q)\leq 1$.
\end{remark}

This lemma gives $H^{(i)}(q,a)$ as a sum of exponentials, and we are thus motivated to investigate their coefficients. The below definition gives a name to these coefficients.
\begin{defn}\label{def:coeffs S}
    Define, for $q_1,q_2,q_3$,
    \begin{align*}
        S^{(1)}({q_1,q_2,q_3})&:=\sum_{\substack{(d_1,d_2,d_3)\in\mathcal G_1\\ q_i\mid d_i}}\frac{\lambda_R(d_1)\lambda_Q(d_2)\rho_T(d_3)}{d_1d_2d_3},\\
        S^{(2)}(q_1,q_2,q_3)&:= \sum_{\substack{(d_1,d_2,d_3)\in \mathcal G_2':\\ q_i\mid d_i}}\frac{\lambda_R(2d_1)\lambda_Q(2d_2)\rho_T(d_3)}{2d_1d_2d_3}.
    \end{align*}
\end{defn}
First observe that since the $d_i$ are pairwise coprime, $q_i\mid d_i\Rightarrow q_i$ are pairwise coprime too, so this is a necessary condition for the coefficients $S^{(i)}(q_1,q_2,q_3)$ to be non-zero. Next, if we write $q=q_\text{free}q_\text{full}$, where
\begin{equation}\label{eq:notation for q' and q''}
    q_\text{full}:=\prod_{\substack{p:\nu_p(q)\ge 2}}p^{\nu_p(q)},\qquad q_\text{free}:=\prod_{\substack{p:\nu_p(q)=1}}p,
\end{equation}
then, as $\lambda_R(q_1),\lambda_Q(q_2)$ are supported on the square-frees, $q_\text{full}\mid q_3$ is necessary for $S^{(i)}(q_1,q_2,q_3)\neq 0$. Thus,
\begin{equation}\label{eq:H1 as a sum of exponentials with coeffs S}
    H^{(1)}(q,a) = \sum_{\substack{q_1q_2q_3 = q:\\ q_\text{full}\mid q_3,\\ q_i \text{ pairwise coprime}}}S^{(1)}(q_1,q_2,q_3)e_{q_1}\left(a\overline{q_2q_3}^{(q_1)}\right)e_{q_2}\left(-a\overline{q_1q_3}^{(q_2)}\right).
\end{equation}
For $S^{(2)}(q_1,q_2,q_3)$, a similar argument goes. Since we can restrict attention to $\nu_2(q)\le 1$, we see that even if $\nu_2(q) = 1$, $m_\text{full} = q_\text{full}$, so we have
\begin{equation}\label{eq:H2 as a sum of exponentials with coeffs S}
    H^{(2)}(q,a)  = \begin{cases}
        \sum\limits_{\substack{q_1q_2q_3 = q:\\ q_\text{full}\mid q_3,\\ q_i \text{ pairwise coprime}}}S^{(2)}(q_1,q_2,q_3)e_{q_1}\qty(a\overline{q_2q_3}^{(q_1)})e_{q_2}\qty(-a\overline{q_1q_3}^{(q_2)}) & \text{if } \nu_2(q) = 0,\\
        e_2(a) \sum\limits_{\substack{q_1q_2q_3 = m:\\ q_\text{full}\mid q_3,\\ q_i \text{ pairwise coprime}}} S^{(2)}(q_1,q_2,q_3) e_{q_1}\qty(a\overline{2q_2q_3}^{(q_1)})e_{q_2}\qty(-a\overline{2q_1q_3}^{(q_2)}) & \text{if } \nu_2(q) = 1,\ q = 2m\\
        0 & \text{if } \nu_2(q)\ge 2
    \end{cases}
\end{equation}

We can now focus on the coefficients $S^{(i)}(q_1,q_2,q_3)$, where the $q_i$ are pairwise coprime and $q_\text{full}\mid q_3$. An inspection of \Cref{def:coeffs S} motivates us to investigate sums of form $\sum_{\ell\mid d} \frac{\lambda_X(d)}{d}$ and $\sum_{\ell\mid d}\frac{\rho_T(d)}d$. The following results address this.

\begin{lem}\label{lambda sums}
For any $\ell\in\mathbb N,\ X\ge 1$, have
\begin{equation*}
    \sum_{\substack{d\leq X:\\ \ell \mid d}}\frac{\lambda_X(d)}{d} = \frac{\mu(\ell)}{\phi(\ell)}\mathbf1_{\ell \leq X},\qquad \sum_{\substack{d\leq X/2:\\ \ell \mid d}}\frac{\lambda_X(2d)}{2d} = -\frac{\mu(\ell)}{\phi(\ell)}\mathbf1_{\substack{2\ell \leq X,\  2\nmid \ell}}.
\end{equation*}
\end{lem}

\begin{proof} Recalling \Cref{def:labmda and rho}, $\lambda_X(d)/d=\sum_{d\mid k\le X}\mu(k)\mu(k/d)/\phi(k)$, so
\begin{equation*}
    \sum_{\substack{d\le X\\ \ell \mid d}}\frac{\lambda_X(d)}d = \sum_{k\le X}\frac{\mu(k)}{\phi(k)}\sum_{\substack{d:\ \ell \mid d\mid k}}\mu\left(\frac{k}{d}\right).
\end{equation*}
 Write $k=\ell k'$, $d=\ell d'$, $d'\mid k'$. Then, we use M\"{o}bius inversion to get
\begin{equation*}
    \sum_{\substack{d:\ \ell \mid d\mid k}}\mu\left(\frac{k}{d}\right) = \sum_{\substack{d'\mid k'}}\mu\left(\frac{k'}{d'}\right)=\mathbf 1_{\{k'=1\}}.
\end{equation*}
So only $k = \ell $ survives, yielding
\begin{equation*}
    \sum_{\substack{d\leq X:\\ \ell \mid d}} \frac{\lambda_X(d)}{d} = \sum_{k\leq X}\frac{\mu(k)}{\phi(k)}\vb{1}_{k=\ell } = \frac{\mu (\ell)}{\phi(\ell)}\vb{1}_{\ell \leq X}.
\end{equation*}
For the second equation, simply use the first equation to get
\begin{equation*}
    \sum_{\substack{d\leq X/2:\\ \ell\mid d}}\frac{\lambda_X(2d)}{2d}=\sum_{\substack{ d\le X\\ 2\ell \mid d}}\frac{\lambda_X(d)}{d} =\frac{\mu(2\ell)}{\phi(2\ell )}\mathbf 1_{2\ell\le X}= -\frac{\mu(\ell)}{\phi(\ell )}\mathbf1_{\substack{2\ell\leq X, 2\nmid \ell}}.\qedhere
\end{equation*}
\end{proof}

In the spirit of this result, we make the below definition.
\begin{defn}[$\gamma({\kappa})$]
For $\kappa\geq 1$, we define
\begin{equation*}
\gamma({\kappa}): = \sum_{\substack{\kappa \mid d}} \frac{\rho_T(d)}{d}.
\end{equation*}
\end{defn}
Now we want to get a handle on this function. First, a lemma.

\begin{lem}\label{Truncated harmonic sum with coprimality}
For $d\ge 1$, $X > 0$,
\begin{equation*}
    L_X(d):= \sum_{\substack{X\le n\le 2X\\ (n,d)=1}}\frac{1}{n} =\frac{\phi(d)}{d}\log 2+\order{\min\qty(1,\frac{\tau(d)}{X})}.
\end{equation*}
\end{lem}
\begin{proof} For $X > 0$, by M\"{o}bius inversion,
\[L_X(d)=\sum_{e\mid d}\mu(e)\sum_{\substack{X\le n\le 2X\\ e\mid n}}\frac{1}{n}=\sum_{e\mid d}\frac{\mu(e)}{e}\qty(\log 2+\order{\frac{e}{X}})=\frac{\phi(d)}{d}\log 2+\order{\frac{\tau(d)}{X}}.\]
and for any $X$, $L_X(d)$ is a sum of at most $X+1$ terms, all at most $1/X$, so is bounded above by $2$.
\end{proof}
In the following proposition, we see that $\gamma(\kappa)$ is well-approximated by a multiplicative function.
\begin{prop}[Product form for $\gamma(\kappa)$]\label{prop:product form for gamma(kappa)}
We have, for $0<\theta<1$,
\begin{equation}\label{error for gamma(kappa) in nonsquarefree case}
    \gamma(\kappa) = \widetilde{\gamma}(\kappa) + \mathcal O_\theta\qty(\frac{1}{T^{\theta}}\qty(\frac{2^{\theta}}{1-2^{\theta-1}})^{3\omega(\kappa)}).
\end{equation}
where
\begin{equation*}
    \widetilde{\gamma}(\kappa)  = (\log 2)^3\prod_{p\mid \kappa}\Xi_{\nu_p(\kappa)} (p),
\end{equation*}
and the local factors are
\begin{equation*}
    \Xi_\ell(p): = \frac{1}{p^{\ell+2}}\qty[\binom{\ell+2}{2}(p-1)^2+(\ell+2)(p-1)+1].
\end{equation*}
\begin{remark}
    Note that $\Xi_0(p) = 1$, and that $\Xi_1(p) = 1-\qty(1-\frac 1p)^3$.
\end{remark}
\end{prop}
\begin{proof}
Let us first introduce some notation.
\begin{notn}
    Given $P = p_1p_2...p_r$ square-free with $p_i$ its prime factors, and $\vb{e}\in\mathbb Z^r$, we write
    \begin{equation*}
        P^{\vb{e}}:= \prod_{i=1}^rp_i^{e_i}.
    \end{equation*}
\end{notn}

Then, writing $\kappa = \prod_{i=1}^r p_i^{k_i}$ (where $k_i>0$) and $\text{rad}(\kappa) := P$, have
\begin{equation*}
    \gamma(\kappa) = \sum_{\kappa\mid d}\frac{\rho_T(d)}{d}= \sum_{\substack{T\leq d_1,d_2,d_3\leq 2T:\\ \kappa\mid d_1d_2d_3   }}\frac{1}{d_1d_2d_3}=\sum_{(\mathbf e_1,\mathbf e_2,\mathbf e_3)\in\mathcal D(\kappa)}\frac{1}{P^{\mathbf e_1+\mathbf e_2+\mathbf e_3}}\prod_{i=1}^3L_{T/P^{\mathbf e_i}}(P),
\end{equation*}
where $\mathcal D(\kappa):=\{(\mathbf e_1,\mathbf e_2,\mathbf e_3)\in\mathbb Z_{\ge 0}^{r}\times \mathbb Z_{\ge 0}^{r}\times \mathbb Z_{\ge 0}^{r}:e_{1,i}+e_{2,i}+e_{3,i}\ge k_i\ \forall i\}$. Using \Cref{Truncated harmonic sum with coprimality}:
\begin{equation}\label{eq:gamma with first error}
    \gamma(\kappa)=\sum_{(\mathbf e_1,\mathbf e_2,\mathbf e_3)\in\mathcal D(\kappa)}\frac{1}{P^{\mathbf e_1+\mathbf e_2+\mathbf e_3}}\prod_{i=1}^3\qty(\frac{\phi(P)}{P}\log 2+\order{\min\qty(1,\frac{P^{\mathbf e_i}\tau(P)}{T})}).
\end{equation}
The main contribution is then
\begin{equation*}
    \widetilde{\gamma}(\kappa)=(\log 2)^3\qty(\frac{\phi(P)}{P})^3\sum_{(\mathbf e_1,\mathbf e_2,\mathbf e_3)\in\mathcal D(\kappa)}\frac{1}{P^{\mathbf e_1+\mathbf e_2+\mathbf e_3}},
\end{equation*}
where, by multiplicativity,
\begin{equation*}
    \sum_{(\mathbf e_1,\mathbf e_2,\mathbf e_3)\in\mathcal D(\kappa)}\frac{1}{P^{\mathbf e_1+\mathbf e_2+\mathbf e_3}}=\prod_{i=1}^r\sum_{\substack{e_{1,i},e_{2,i},e_{3,i}\ge 0\\ e_{1,i}+e_{2,i}+e_{3,i}\ge k_i}}\frac{1}{p_i^{e_{1,i}+e_{2,i}+e_{3,i}}}.
\end{equation*}
\Cref{probability lemma} gives the inner sum, so
\begin{equation*}
    \widetilde{\gamma}(\kappa)=(\log 2)^3\prod_{i=1}^r\Xi_{k_i}(p_i)=(\log 2)^3\prod_{p\mid\kappa}\Xi_{\nu_p(\kappa)}(p).
\end{equation*}

Since $\min\qty(1,\frac{P^{\vb{e}_i2^{\omega(\kappa)}}}{T})\leq 1$, the error term in \eqref{eq:gamma with first error} at each $(\mathbf e_1,\mathbf e_2,\mathbf{e}_3)$ is
    \[\order{\sum_{i=1}^3\min\qty(1,\frac{P^{\mathbf e_i}\tau(P)}{T})},\]
so the total error term is, by symmetry,
\[\order{\sum_{(\mathbf e_1,\mathbf e_2,\mathbf e_3)\in\mathcal D(\kappa)}\frac{1}{P^{\mathbf e_1+\mathbf e_2+\mathbf e_3}}\min\qty(1,\frac{P^{\mathbf e_1}\tau(P)}{T})}.\]
Using Rankin's trick, for $0<\theta<1$,
\[\min\qty(1,\frac{P^{\mathbf e_1}\tau(P)}{T})\le\qty(\frac{P^{\mathbf e_1}\tau(P)}{T})^\theta,\]
so letting the $\vb{e}_i$ run freely on $\mathbb Z_{\ge 0}^r$, the error term is at most
\begin{equation*}
    \begin{aligned}
        \sum_{\vb{e}_1,\vb{e}_2,\vb{e}_3}\frac{1}{P^{\vb{e}_1 + \vb{e_2}+ \vb{e_3}}}\qty(\frac{P^{\vb{e_1}}\tau(P)}{T})^\theta & = \qty(\frac{\tau(P)}{T})^\theta\qty(\sum_{\vb{e}}\frac{1}{P^{\vb{e}}})^2\sum_{\vb{e}}\frac{1}{P^{(1-\theta)\vb{e}}}\\
        & \ll 2^{\theta\omega(\kappa)}T^{-\theta}\qty(\prod_{p\mid P}\frac{1}{1-p^{-(1-\theta)}})^3\\
        & \ll T^{-\theta}\qty(\frac{2^\theta}{1-2^{-(1-\theta)}})^{3\omega(\kappa)}.
    \end{aligned}
\end{equation*}
\end{proof}
Now that we have these results, we may begin to manipulate $S^{(i)}(q_1,q_2,q_3)$. If we recall \Cref{def:coeffs S} for these coefficients, we see that we would like to ``untangle" the $d_i$ in the summation, which leads us to the following proposition.

\begin{prop}[Alternative form for $S^{(i)}(q_1,q_2,q_3)$]
For any $q_i$, we have
\begin{equation*}
    \begin{aligned}
        S^{(1)}({q_1,q_2,q_3})& = \sum_{u,v,r}\mu(u)\mu(v)\mu(r)\frac{\mu([q_1,u,v])}{\phi([q_1,u,v])}\frac{\mu([q_2,u,r])}{\phi([q_2,u,r])}\gamma({[q_3,r,v]})\mathbf 1_{[q_1,u,v]\leq R}\mathbf1_{[q_2,u,r]\leq Q}.
    \end{aligned}
\end{equation*}
Further, for $q_1,q_2$ square-free and odd, $q_3$ odd:
    \begin{equation*}
        \begin{aligned}
            &S^{(2)}(q_1,q_2,q_3)\\
            &= 2\sum_{u,v,r\text{ odd}}\mu(u)\mu(v)\mu(r)\frac{\mu([q_1,u,v])}{\phi([q_1,u,v])}\frac{\mu([q_2,u,r])}{\phi([q_2,u,r])}({\gamma}({[q_3,r,v]})-{\gamma}(2[q_3,r,v])\vb 1_{[q_1,u,v]\leq R/2}\vb 1_{[q_2,u,r]\leq Q/2}.
        \end{aligned}
    \end{equation*}
\end{prop}
\begin{remark}
    In the second part, the assumptions could be relaxed, but this is not needed, since \eqref{eq:H2 as a sum of exponentials with coeffs S} tells us that we are only concerned with odd $q_i$ when it comes to $S^{(2)}(q_1,q_2,q_3)$.
\end{remark}
\begin{proof}
    For $S^{(1)}(q_1,q_2,q_3)$, we use M\"{o}bius inversion and \Cref{prop:product form for gamma(kappa)}:
    \begin{align*}
        S^{(1)}(q_1,q_2,q_3) & =\sum_{\substack{d_1\le R,\\ d_2\leq Q,\\ d_3\leq 8T^3}}\frac{\lambda_R(d_1)\lambda_Q(d_2)\rho_T(d_3)}{d_1d_2d_3}\vb 1_{(d_1,d_2) = 1}\vb 1_{(d_2,d_3) = 1} \vb 1_{(d_1,d_3) = 1}\\
            & = \sum_{\substack{q_1\mid d_1\le R,\\ q_2\mid d_2\leq Q,\\ q_3\mid d_3\leq 8T^3}}\frac{\lambda_R(d_1)\lambda_Q(d_2)\rho_T(d_3)}{d_1d_2d_3}\sum_{u\mid (d_1,d_2)}\sum_{v\mid (d_1,d_3)} \sum_{r\mid (d_2,d_3)}\mu(u)\mu(v)\mu(r)\\
            &= \sum_{u,v,r}\mu(u)\mu(v)\mu(r)\sum_{\substack{d_1\le R:\\ u,v,q_1\mid d_1}}\frac{\lambda_R(d_1)}{d_1}\sum_{\substack{d_2\leq Q:\\ q_2,u,r\mid d_2}}\frac{\lambda_Q(d_2)}{d_2}\sum_{\substack{d_3\leq 8T^3:\\ q_3,v,r\mid d_3}}\frac{\rho_T(d_3)}{d_3}.
    \end{align*}
    Noting that $u,v,q_1\mid d_1\Leftrightarrow [u,v,q_1]\mid d_1$ (and likewise for the other sums), get by the first part of \Cref{lambda sums} that
    \begin{equation*}
        S^{(1)}(q_1,q_2,q_3) = \sum_{u,v,r}\mu(u)\mu(v)\mu(r)\frac{\mu([q_1,u,v])}{\phi([q_1,u,v])}\frac{\mu([q_2,u,r])}{\phi([q_2,u,r])}\gamma([q_3,v,r])\vb 1_{[q_1,u,v]\leq R}\vb 1_{[q_2,u,r]\leq Q},
    \end{equation*}as needed.

    For $S^{(2)}(q_1,q_2,q_3)$, the proof proceeds analogously, only this time we use the second part of \Cref{lambda sums}; omitting the algebraic manipulations similar to before, get
    \begin{equation*}
        \begin{aligned}
            & S^{(2)}(q_1,q_2,q_3)\\
            & = 2\sum_{u,v,r}\mu(u)\mu(v)\mu(r)\underbrace{\sum_{\substack{d_1\leq R/2:\\ [u,v,q_1]\mid d_1}}\frac{\lambda_R(2d_1)}{2d_1}}_{-\frac{\mu([u,v,q_1])}{\phi([u,v,q_1])}1_{[u,v,q_1]\leq R/2}1_{2\nmid [u,v,q_1]}}\underbrace{\sum_{\substack{d_2\leq Q/2:\\ [u,r,q_2]\mid d_2}}\frac{\lambda_Q(2d_2)}{2d_2}}_{-\frac{\mu([u,r,q_2])}{\phi([u,r,q_2])}1_{[u,r,q_2]\leq Q/2}1_{2\nmid [u,r,q_2]}}\sum_{\substack{d_3\text{ odd}:\\ [q_3,r,v]\mid d_3}}\frac{\rho_T(d_3)}{d_3}.
        \end{aligned}
    \end{equation*}
    Note that since $q_1,q_2$ are odd, we get that $2\nmid [u,v,q_1],2\nmid [u,r,q_2]\Leftrightarrow 2\nmid u,v,r$, so the sum is then over $u,v,r$ odd. Hence, we have
    \begin{equation*}
        \sum_{\substack{d_3\text{ odd}:\\ [q_3,r,v]\mid d_3}}\frac{\rho_T(d_3)}{d_3} = \sum_{\substack{d_3:\\ [q_3,r,v]\mid d_3}}\frac{\rho_T(d_3)}{d_3}-\sum_{\substack{d_3:\\ 2[q_3,r,v]\mid d_3}}\frac{\rho_T(d_3)}{d_3} = \gamma([q_3,r,v])-\gamma(2[q_3,r,v]),
    \end{equation*}
    so are done.
\end{proof}
Since $\widetilde{\gamma}$ is an approximation for $\gamma$, this leads us naturally to the following definitions.
\begin{defn}[Main and error terms for $S^{(i)}(q_1,q_2,q_3)$]\label{def:main and error terms for S coeffs}
Define (will only use for $q_i$ pairwise coprime and $q_\text{full}\mid q_3$):
\begin{equation*}
    S^{(1)}(q_1,q_2,q_3) = S_{\text{main}}^{(1)}(q_1,q_2,q_3) + E^{(1)}_1(q_1,q_2,q_3) + E_{2}^{(1)}(q_1,q_2,q_3),
\end{equation*}where the main terms are the untruncated sums with $\gamma$ replaced by $\widetilde{\gamma}$:
\begin{equation}\label{eq:first main term}
    S_{\text{main}}^{(1)}(q_1,q_2,q_3) =  \sum_{u,v,r}\mu(u)\mu(v)\mu(r)\frac{\mu([q_1,u,v])}{\phi([q_1,u,v])}\frac{\mu([q_2,u,r])}{\phi([q_2,u,r])}\widetilde{\gamma}([q_3,v,r]).
\end{equation}
and the error terms $E_1^{(1)}(q_1,q_2,q_3)$ are given by the loss in approximating $\gamma$:
\begin{equation*}
    \qty|E_1^{(1)}(q_1,q_2,q_3)| \leq \sum_{u,v,r}\mu(u)^2\mu(v)^2\mu(r)^2\frac{|\gamma([q_3,v,r])-\widetilde{\gamma}([q_3,v,r])|}{\phi([q_1,u,v])\phi([q_2,u,r])}\mathbf1_{[q_1,u,v]\leq R}\mathbf1_{[q_2,u,r]\leq Q}.
\end{equation*}
and the error terms $E_{2}^{(1)}(q_1,q_2,q_3)$ arise from the truncation:
\begin{align*}
    \qty|E_2^{(1)}(q_1,q_2,q_3)|&\leq \sum_{u,v,r}\mu(u)^2\mu(v)^2\mu(r)^2\frac{\widetilde{\gamma}([q_3,v,r])}{\phi([q_1,u,v])\phi([q_2,u,r])}\qty(1-\mathbf 1_{[q_1,u,v]\leq R}\mathbf 1_{[q_2,u,r]\leq Q}).
\end{align*}
For $q_1,q_2,q_3$ square-free and odd, we define analogously
    \begin{equation*}
        S^{(2)}(q_1,q_2,q_3) = S_\text{main}^{(2)}(q_1,q_2,q_3) + E_1^{(2)}(q_1,q_2,q_3) + E_2^{(2)}(q_1,q_2,q_3).
    \end{equation*}
\end{defn}
We make these definitions because now the main terms can be factored as Euler products, as seen in the following proposition.
\begin{prop}[Main term for $S^{(i)}(q_1,q_2,q_3)$]\label{prop:main terms for S}
    For $q_\text{full}\mid q_3$ and $q = q_1q_2q_3$, $q_i$ pairwise coprime, we have
    \begin{equation*}
        S_{\text{main}}^{(1)}(q_1,q_2,q_3) = C(q)h(q_1q_2)k(q_3),
    \end{equation*}
    where $h,k$ are multiplicative functions with
    \begin{equation*}
        h(p)=-\frac{p-1}{p^2},\qquad k\qty(p^\ell)=\frac{p^2}{(p-1)^2}\Xi_{\ell}(p)
    \end{equation*}
    at each prime $p$ and $\ell \ge 1$, and
    \begin{equation*}
        C(q):=(\log 2)^3\prod_{p\nmid q}\qty(1+\frac{5p^2-6p+2}{p^2(p-1)^2})
    \end{equation*}is a constant with respect to $(q_1,q_2,q_3)$.
    
    If, further, $q_i$ are odd, we have
    \begin{equation*}
        S_\text{main}^{(2)}(q_1,q_2,q_3) = \frac 1{14} S_\text{main}^{(1)}(q_1,q_2,q_3).
    \end{equation*}
\end{prop}
\begin{remark}
    We know that $C(q)$ is indeed a convergent Euler product, because $\frac{5p^2-6p+2}{p^2(p-1)^2} = \order{p^{-2}}$, and $\sum_p p^{-2}<\infty$.

    Additionally, observe that from $q_\text{full}\mid q_3$ we necessarily have that $q_1,q_2$ are square-free.
\end{remark}

\begin{proof}
    Recall \eqref{eq:first main term} and see that the summands are supported on $u,v,r$ square-free. So, we can let $u_p:=\nu_p(u)$ --- and likewise define $v_p$, $r_p$ --- to get that
    \begin{align*}
        S_\text{main}^{(1)}(q_1,q_2,q_3)/(\log 2 )^3 & = \sum_{u,v,r}\prod_{p} \mathfrak H(p,u_p,v_p,r_p) = \prod_p\sum_{u_p,v_p,r_p\in\{0,1\}}\mathfrak H(p,u_p,v_p,r_p)  ,\\
        \text{where}\quad \mathfrak H(p,u_p,v_p,r_p) & := (-1)^{u_p+v_p+r_p}\cdot \begin{cases}
            \qty(\frac{-1}{p-1})^{\max(u_p,v_p) + \max(u_p,r_p)}\cdot \Xi_1(p)^{\max(v_p,r_p)} & \text{ if } p\nmid q\\
            \qty(\frac{-1}{p-1})^{1 + \max(u_p,r_p)}\cdot\Xi_1(p)^{\max(v_p,r_p)} & \text{ if } p\mid q_1\\
            \qty(\frac{-1}{p-1})^{1+\max(u_p,v_p)}\cdot\Xi_1(p)^{\max(v_p,r_p)} & \text{ if } p\mid q_2\\
            \qty(\frac{-1}{p-1})^{\max(u_p,r_p) + \max(u_p,v_p)}\cdot \Xi_{\nu_p(q_3)}(p) & \text{ if } p\mid q_3.
        \end{cases}
    \end{align*}

    We can now compute $\sum\limits_{u_p,v_p,r_p}\mathfrak H(p,u_p,v_p,r_p)$ by noting that it is simply the sum of $2^3 = 8$ terms. Leaving the calculation as an exercise to the reader, the results are
    \begin{equation*}
        \sum_{u_p,v_p,r_p}\mathfrak H(p,u_p,v_p,r_p) =  \begin{cases}
            1 +\frac{5p^2-6p+2}{p^2(p-1)^2} & \text{ if } p\nmid q \\
            -\frac{p-1}{p^2} & \text{ if } p\mid q_1\text{ or } p\mid q_2\\
            \frac{p^2}{(p-1)^2}\Xi_{\nu_p(q_3)}(p) & \text{ if } p\mid q_3
        \end{cases} = \begin{cases}
            1 +\frac{5p^2-6p+2}{p^2(p-1)^2} & \text{ if } p\nmid q \\
            h(p) & \text{ if } p\mid q_1\text{ or } p\mid q_2\\
            k\qty(p^{\nu_p(q_3)}) & \text{ if } p\mid q_3.
        \end{cases}
    \end{equation*}
So we are done.

For $S^{(2)}(q_1,q_2,q_3)$, inspect the above to see that, for the same definition of $\mathfrak H$, we have
    \begin{equation*}
        S_\text{main}^{(2)}(q_1,q_2,q_3)/(\log 2)^3 = 2\cdot\frac 18\sum_{u,v,r\text{ odd}}\prod_{p\neq 2}\mathfrak H(p,u_p,v_p,r_p) = \frac 14\prod_{p\neq 2}\sum_{u_p,v_p,r_p}\mathfrak H(p,u_p,v_p,r_p)
    \end{equation*}
    where the factor of $\frac 18$ arises from the fact that, for $n$ odd, $\widetilde{\gamma}(n) - \widetilde{\gamma}(2n) = \widetilde{\gamma}(n)(1-\Xi_1(2)) = \frac{1}{8}\widetilde{\gamma}(n).$
    
    So we simply have
    \begin{equation*}
        S_\text{main}^{(2)} (q_1,q_2,q_3) = \frac 14\cdot\eval{\qty(1 + \frac{5p^2-6p+2}{p^2(p-1)^2})^{-1}}_{p=2}S_\text{main}^{(1)}(q_1,q_2,q_3) = \frac{1}{14} S_\text{main}^{(1)}(q_1,q_2,q_3),
    \end{equation*} as needed.
\end{proof}
Now we wish to control the error terms. Since $E_1^{(i)}(q_1,q_2,q_3)$ arises from the error in estimating $\gamma$, the size of which we have already bounded, we are not going to struggle too much.

\begin{lem}[Bound on $E_{1}^{(i)}(q_1,q_2,q_3)$]\label{bound on E(1)}
For $0<\theta<1$, $i = 1,2$, we have
\begin{equation*}
    \qty|E_1^{(i)}(q_1,q_2,q_3)|\ll_{\theta} \log ^{\order{1}}N\frac{q_3^{o(1)}}{T^{\theta}}.
\end{equation*}
\end{lem}
\begin{proof}
    For $i = 1$, use \eqref{error for gamma(kappa) in nonsquarefree case} and the crude bounds
    \begin{equation*}
        \phi([q_1,u,v])\geq \phi([u,v]),\quad \phi([q_2,u,r])\geq \phi([u,r]),\quad \omega([q_3,v,r])\leq \omega(q_3)+\omega(v)+\omega(r)
    \end{equation*}
    to get
    \begin{equation*}
        \begin{aligned}
            |E_1^{(1)}(q_1,q_2,q_3)| & \ll_\theta \frac{1}{T^{\theta}}\sum_{\substack{u,v,r:\\ [u,v]\leq R,\\ [u,r]\leq Q}}\frac{\mu(u)^2\mu(v)^2\mu(r)^2}{\phi([u,v])\phi([u,r])}\qty(\frac{2^\theta}{1-2^{\theta-1}})^{3\omega([q_3,v,r])}\\
            & \ll_\theta \frac{1}{T^{\theta}}q_3^{o(1)}\sum_{\substack{u,v,r\:\\ [u,v]\leq R, [u,r]\leq Q}} \frac{\mu(u)^2\mu(v)^2\mu(r)^2}{\phi([u,v])\phi([u,r])}\tau(v)^{\mathcal O_\theta(1)}\tau(r)^{\mathcal O_\theta(1)}
        \end{aligned}
    \end{equation*}
    Here, we used that for $\ell:=\frac{2^\theta}{1-2^{\theta - 1}}$, $\ell>1$, and that by the divisor bound, $\ell^{3\omega(x)}= 2^{(3\log_2 \ell) \omega(x)}\ll \tau(x)^{\order{1}}\ll x^{o(1)}$, where we specialise for $x = q_3,v,r$.
    
    Then, we look at the sum over $v$: for $u$ square-free and fixed, $B\geq 0$ any constant, write $v = v_1v_2$, where $v_1 = (u,v)$, so $(v_2,u) = 1$ and $[u,v] = uv_2$. Using \Cref{lem:ben green b.2},
    \begin{equation*}
        \begin{aligned}
            \sum_{v:[u,v]\leq R}\frac{\mu(v)^2}{\phi([u,v])}\tau(v)^B & \le \frac{1}{\phi(u)} \sum_{v_1\mid u}\mu(v_1)^2\tau(v_1)^B\sum_{\substack{v_2\leq R/u:\\ (u,v_2) = 1}}\frac{\mu(v_2)^2\tau(v_2)^B}{\phi(v_2)}\\
            & \ll \frac{1}{\phi(u)}\tau(u)^{B+1}\sum_{v_2\leq R/u}\frac{\mu(v_2)^2\tau(v_2)^B}{v_2}\log \log v_2\\
            & \ll \frac{u^{o(1)}}{\phi(u)}\log \log \qty(\frac R u)\log ^{\order{1}}\qty(\frac R u)\ll \frac{u^{o(1)}}{\phi(u)}\log^{\order{1}} N.
        \end{aligned}
    \end{equation*}
    Similarly, we get that the sum over $r$ has the same bound, so
    \begin{equation*}
        \sum_{\substack{u,v,r\:\\ [u,v]\leq R, [u,r]\leq Q}} \frac{\mu(u)^2\mu(v)^2\mu(r)^2}{\phi([u,v])\phi([u,r])}\tau(v)^{\mathcal O_\theta(1)}\tau(r)^{\mathcal O_\theta(1)}\ll \log^{\order{1}} N\sum_{u}\frac{u^{o(1)}}{\phi(u)^2}\ll \log ^{\order{1}}N.
    \end{equation*}
    This concludes the proof for the $i = 1$ case; the $i = 2$ case follows completely analogously; the only difference is that
    \begin{equation*}
        \qty|E_1^{(2)}(q_1,q_2,q_3)| \leq \sum_{u,v,r}\frac{|\gamma([q_3,v,r])-\widetilde{\gamma}([q_3,v,r])| + \abs{\gamma(2[q_3,v,r])-\widetilde{\gamma}(2[q_3,v,r])}}{\phi([q_1,u,v])\phi([q_2,u,r])}\mathbf1_{[q_1,u,v]\leq R}\mathbf1_{[q_2,u,r]\leq Q}
    \end{equation*}
    and that $\tau(2[q_3,v,r])\ll \tau([q_3,v,r])$.
\end{proof}

We now move onto the second error term.
\begin{prop}[Bound on $E_{2}^{(i)}(q_1,q_2,q_3)$]\label{bound on E(2)}
Suppose $Q\ll R$. Then, for $\varepsilon>0$, $0<\theta<1$ and $q= q_1q_2q_3$, where $q_1,q_2$ are square-free and $q_3$ is not necessarily so,
\begin{align*}
    \qty|E_{2}^{(1)}(q_1,q_2,q_3)|&\ll_{\theta,\varepsilon} \frac{q_3^\varepsilon}{q}\qty(\qty(\frac {q_1}R)^\theta + \qty(\frac{q_2}{Q})^\theta) \ll \frac{q^\varepsilon}q\qty(\frac qQ)^\theta.
\end{align*}
Further, for $q=q_1q_2q_3$ odd (and $q_1,q_2$ square-free), and $0<\theta<1$, we get the same bound:
\begin{equation*}
    \abs{E_2^{(2)}(q_1,q_2,q_3)} \ll_{\varepsilon,\theta} \frac {q_3^\varepsilon}{q}\qty[\qty(\frac {q_1}{R})^\theta + \qty(\frac{q_2}{Q})^\theta] \ll \frac {q^\varepsilon}{q}\qty(\frac qQ)^\theta.
\end{equation*}
\end{prop}
\begin{proof}
    We use Rankin's trick to get that for all $0<\theta<1$:
    \begin{equation*}
        1-\vb 1_{[q_1,u,v]\leq R}\vb 1_{[q_2,u,r]\leq Q}\leq \qty(\frac{[q_1,u,v]}{R})^\theta + \qty(\frac{[q_2,u,r]}{Q})^\theta,
    \end{equation*}
    and so, recalling \Cref{def:main and error terms for S coeffs}, we just need to evaluate the sum
    \begin{equation*}
        \sum_{u,v,r}\frac{\widetilde{\gamma}([q_3,v,r])}{\phi([q_1,u,v])\phi([q_2,u,r])}\qty(\qty(\frac{[q_1,u,v]}{R})^\theta + \qty(\frac{[q_2,u,r]}{Q})^\theta)
    \end{equation*}
    This splits into two sums, which we can evaluate in the same way as we did in the proof for the main terms in \Cref{prop:main terms for S}, using multiplicativity.
    
    For brevity, we omit the details of the calculation, but see that we end up with
    \begin{equation*}
        \begin{aligned}
            \sum_{u,v,r}\frac{\widetilde{\gamma}([q_3,v,r])}{\phi([q_1,u,v])\phi([q_2,u,r])}\qty(\frac{[q_1,u,v]}{R})^\theta & =\frac{(\log 2)^3}{R^\theta} \prod_{p} \mathcal F(p)\prod_{p\mid q_1}\frac{\mathfrak f_1(p)}{\mathcal F(p)}\prod_{p\mid q_2}\frac{\mathfrak f_2(p)}{\mathcal F(p)}\prod_{p\mid q_3}\frac{\mathfrak f_3(p)}{\mathcal F(p)},
        \end{aligned}
    \end{equation*}
    where
    \begin{equation*}
        \begin{aligned}
             \mathcal F(p) & = 1 + \Xi_1(p)\qty(\frac{4p^\theta}{(p-1)^2} + \frac{p^\theta}{p-1}+\frac{1}{p-1})+\frac{p^\theta}{(p-1)^2},\\
             \mathfrak f_1(p) & = \frac{p^\theta}{p-1}\qty(1 + \Xi_1(p)\qty(\frac{5}{p-1}+1)+\frac{1}{p-1})\ll \frac{p^\theta}{p},\\
             \mathfrak f_2(p)& = \frac{1}{p-1}\qty(1 + \Xi_1(p)\qty(\frac{5p^\theta}{p-1}+ 1)+\frac{p^\theta}{p-1})\ll \frac 1p,\\
             \mathfrak f_3(p) & = \Xi_{\nu_p(q_3)}(p)\qty(1 + \frac{5p^\theta}{p-1} + \frac{1}{p-1}+\frac{p^\theta}{p-1})\ll \frac{\nu_p(q_3)^2}{p^{\nu_p(q_3)}}.
        \end{aligned}
    \end{equation*}
    Note that the condition $\theta\in(0,1)$ is indeed necessary for the infinite product to converge:
    \begin{equation*}
        \begin{aligned}
            \mathcal F(p) & = 1 + \order{p^{\theta - 2}},\quad \text{ and }\quad \sum_p p^{-(2-\theta)}\leq \sum_{n} n ^ {-(2-\theta)}<\infty.
        \end{aligned}
    \end{equation*}
    Also observe that
    \begin{equation*}
        \prod_{p\mid q_3} \nu_p(q_3)^2 \leq \qty(\prod_{p\mid q_3}(1+\nu_p(q_3)))^2 = \tau(q_3)^2\ll_\varepsilon q_3^\varepsilon.
    \end{equation*}
    Further, $\mathcal F(p) \ge 1$, so we get
    \begin{equation*}
        \begin{aligned}
            \sum_{u,v,r}\frac{\widetilde{\gamma}([q_3,v,r])}{\phi([q_1,u,v])\phi([q_2,u,r])}\qty(\frac{[q_1,u,v]}{R})^\theta & \ll_\varepsilon \frac{1}{R^\theta}\frac1qq_1^\theta q_3^\varepsilon\prod_p\mathcal F(p)\ll \frac {q_3^\varepsilon}q \frac{q_1^\theta}{R^\theta}.
        \end{aligned}
    \end{equation*}
    For the second sum, we proceed similarly to get that (with the same definitions of $\mathcal F, \mathfrak f_i$):
    \begin{equation*}
        \begin{aligned}
            \sum_{u,v,r}\frac{\widetilde{\gamma}([q_3,v,r])}{\phi([q_1,u,v])\phi([q_2,u,r])}\qty(\frac{[q_2,u,r]}{Q})^\theta & = \frac{(\log 2)^3}{Q^\theta} \prod_{p} \mathcal F(p)\prod_{p\mid q_1}\frac{\mathfrak f_2(p)}{\mathcal F(p)}\prod_{p\mid q_2}\frac{\mathfrak f_1(p)}{\mathcal F(p)}\prod_{p\mid q_3}\frac{\mathfrak f_3(p)}{\mathcal F(p)}\ll_\varepsilon \frac{q_3^\varepsilon}{q}\frac{q_2^\theta}{Q^\theta}.
        \end{aligned}
    \end{equation*}
    For $E_2^{(2)}(q_1,q_2,q_3)$, we have that
    \begin{equation*}
        \begin{aligned}
            \qty|E_2^{(2)}(q_1,q_2,q_3)|&\leq \sum_{u,v,r\text{ odd }}\frac{\widetilde{\gamma}([q_3,v,r]) - \widetilde{\gamma}(2[q_3,v,r])}{\phi([q_1,u,v])\phi([q_2,u,r])}\qty(1-\mathbf 1_{[q_1,u,v]\leq R/2}\mathbf 1_{[q_2,u,r]\leq Q/2})\\
            & = \frac 18\sum_{u,v,r\text{ odd }}\frac{\widetilde{\gamma}([q_3,v,r])}{\phi([q_1,u,v])\phi([q_2,u,r])}\qty(1-\mathbf 1_{[q_1,u,v]\leq R/2}\mathbf 1_{[q_2,u,r]\leq Q/2}),
        \end{aligned}
    \end{equation*}
    and the rest of the proof should proceed analogously.
\end{proof}
We have all the ingredients necessary to go back to \eqref{eq:H1 as a sum of exponentials with coeffs S} and \eqref{eq:H2 as a sum of exponentials with coeffs S} and substitute in our main terms for $S^{(i)}(q_1,q_2,q_3)$.

For the main term of $H^{(i)}(q,a)$, anticipating the algebra to follow, we present the following proposition.
\begin{prop}\label{prop:multiplicative sums}
For any $q$,
\begin{equation}\label{eq:exponential phases interacting with S1main}
    \begin{aligned}
        \sum_{\substack{q_1q_2q_3=q\\ q_\text{full}\mid q_3,\\ q_i\text{ pairwise coprime}}}&e_{q_1}\qty(a\overline{q_2q_3}^{(q_1)})e_{q_2}\qty(-a\overline{q_1q_3}^{(q_2)})h(q_1q_2)k(q_3)\\=&\frac{1}{q}\prod_{p:\nu_p(q)\ge2}\Theta_{\nu_p(q)}(p)\prod_{p:\nu_p(q)=1}\qty(\frac{3p^2-3p+1}{(p-1)^2}-2\frac{p-1}{p}\cos\qty(\frac{2\pi a\overline{(q/p)}^{(p)}}{p})),
    \end{aligned}
\end{equation}
where the local factors arising from the prime factors of $q_\text{full}$ are given by
    \begin{equation*}
        \Theta_\ell(p): = \binom{\ell+2}{2}+(\ell+2)(p-1)^{-1}+(p-1)^{-2} = \frac{p^{\ell + 2}}{(p-1)^2}\Xi_\ell(p).
    \end{equation*}
Further, for $m = q/2$ odd,
\begin{equation}\label{eq:exponential phases interact with S2main}
    \begin{aligned}
        \sum_{\substack{q_1q_2q_3=m\\ q_\text{full}\mid q_3,\\ q_i\text{ pairwise coprime}}}e_{q_1}&\qty(a\overline{2q_2q_3}^{(q_1)})e_{q_2}\qty(-a\overline{2q_1q_3}^{(q_2)})h(q_1q_2)k(q_3)=\\ &\frac{1}{m}\prod_{p:\nu_p(m)\ge2}\Theta_{\nu_p(q)}(p)
        \prod_{p:\nu_p(m)=1}\qty(\frac{3p^2-3p+1}{(p-1)^2}-2\frac{p-1}{p}\cos\qty(\frac{2\pi a\overline{(2m/p)}^{(p)}}{p})).
    \end{aligned}
\end{equation}
\end{prop}

\begin{proof}
    Using the Chinese remainder theorem, get that for any square-free $\ell$ and coefficients $s_p:=\overline{\ell / p}^{(p)}$,
    \begin{equation*}
        \frac 1\ell = \sum_{p\mid \ell} \frac{s_p}{p}\mod 1 \Leftrightarrow 1 \equiv \sum_{p\mid \ell} \qty(\frac \ell p)s_p\mod \ell \Leftrightarrow 1 \equiv \qty(\frac \ell p)s_p\mod p\quad \forall p\mid \ell,
    \end{equation*}
    which manifestly holds. So
    \begin{equation*}
        e_\ell(x) = \prod_{p\mid \ell}e_p\qty(x\overline{(\ell / p)}^{(p)}).
    \end{equation*}
    Since we have $q_1,q_2$ square-free in the sum we are considering, get
    \begin{equation*}
    \begin{aligned}
        e_{q_1}\left(a\overline{q_2q_3}^{(q_1)}\right)=\prod_{p\mid q_1}e_{p}\left(a\overline{q_2q_3}^{(q_1)}\overline{(q_1/p)}^{(p)}\right)&=\prod_{p\mid q_1}e_p\qty(a\overline{(q/p)}^{(p)}),\\
        \text{ and likewise }\quad  e_{q_2}\left(-a\overline{q_1q_3}^{(q_2)}\right)&=\prod_{p\mid q_2}e_p\left(-a\overline{(q/p)}^{(p)}\right).
    \end{aligned}
\end{equation*}
Recalling our notation for $q_\text{free},q_\text{full}$ in \eqref{eq:notation for q' and q''}, we can set $q_3 = q_\text{full}\cdot q_3'$, turning the left hand side of \eqref{eq:exponential phases interacting with S1main} into a (multiplicative) convolution of multiplicative functions, with a factor of $k(q_\text{full})$:
\begin{equation*}
    \begin{aligned}
        & k(q_\text{full})\sum_{\substack{q_1q_2q_3' = q_\text{free},\\ q_1,q_2,q_3'\text{ pairwise coprime}}} \prod_{p\mid q_1}\qty(e_p\qty(a\overline{(q_1/p)}^{(p)})h(p))\prod_{p\mid q_2}\qty(e_p\qty(a\overline{(q_2/p)}^{(p)})h(p)) \prod_{p\mid q_3'}k\qty(p)\\
        &= \prod_{p: \nu_p(q) \ge 2} k(p^{\nu_p(q)}) \prod_{p:\nu_p(q) = 1}\qty(k(p) +h(p)\qty(e_p\qty(a\overline{(q/p)}^{(p)}) + e_p\qty(-a\overline{(q/p)}^{(p)})))
    \end{aligned}
\end{equation*}
Opening the definitions of $h,k$ from \Cref{prop:main terms for S} then gives \eqref{eq:exponential phases interacting with S1main}, as needed.

For \eqref{eq:exponential phases interact with S2main}, compare to the proof above: the only change we make is to note that
    \begin{equation*}
        e_{q_1}\qty(a\overline{2q_2q_3}^{(q_1)}) = \prod_{p\mid q_1}e_p\qty(a\overline{2q_2q_3}^{(q_1)}\overline{q_1/p}^{(p)}) = \prod_{p\mid q_1}e_p\qty(a\overline{(q/p)}^{(p)})
    \end{equation*}
    and likewise for $e_{q_2}\qty(a\overline{2q_1q_3}^{(q_2)})$.
\end{proof}

We are now ready to combine everything and prove \Cref{prop:decomp for H}.

\begin{proof}[Proof of \Cref{prop:decomp for H}]
    Let us deal with the error term first. Substituting \Cref{def:main and error terms for S coeffs} into \eqref{eq:H1 as a sum of exponentials with coeffs S} and \eqref{eq:H2 as a sum of exponentials with coeffs S} yields
\begin{equation*}
    \abs{\mathcal E} = \order{\sum_{\substack{q_1q_2q_3=q:\\ q_\text{full}\mid q_3}}E_1^{(1)}(q_1,q_2,q_3) + E_2^{(1)}(q_1,q_2,q_3) + E_1^{(2)}(q_1,q_2,q_3) + E_2^{(2)}(q_1,q_2,q_3)}
\end{equation*}
We use our error bounds from \Cref{bound on E(1),bound on E(2)}, as well as the crude bound that the above has $\tau_3(q)$ summands, to get that for all $0<\theta<1,$
\begin{equation*}
    \abs{\mathcal E} \ll_{\theta} \tau_3(q)\qty(\frac{q^{o(1)}}{q}\qty(\frac qQ)^\theta + \log^{\order{1}}N \frac{q^{o(1)}}{T^{\theta}})
\end{equation*}
Using that $\tau_3(q)\ll q^{o(1)}\ll N^{o(1)}$, get
\begin{equation*}
    \abs{\mathcal E} \ll_{\theta} N^{o(1)}\cdot\qty(\qty(\frac{1}{Q^\theta q^{1-\theta}}) + \frac{1}{T^{\theta}}) \ll N^{o(1)}\qty(\frac{1}{Q^\theta}+\frac{1}{T^{\theta}}).
\end{equation*}

Then for the main term, first note that when $i = 1$, we do not have to split cases for $q$. Using \Cref{prop:multiplicative sums} and our expression for the main term from \Cref{prop:main terms for S}, we get that for all $q$,
\begin{equation}\label{eq:final form for H01}
    \begin{aligned}
        H_0^{(1)}(q,a) = \frac{C(q)}{q}\prod_{p:\nu_p(q)\ge2}\Theta_{\nu_p(q)}(p)\prod_{p:\nu_p(q)=1}\qty(\frac{3p^2-3p+1}{(p-1)^2}-2\frac{p-1}{p}\cos\qty(\frac{2\pi a\overline{(q/p)}^{(p)}}{p})).
    \end{aligned}
\end{equation}
Then, recall \eqref{eq:H2 as a sum of exponentials with coeffs S} to see that when $i = 2$, we need to split into the three cases $\nu_2(q) = 0,1$ and $\nu_2(q) \ge 2$. Use the same argument to get that, for any $a\in\mathbb Z$,
\begin{equation}\label{eq:final form for H02}
    H_0^{(2)}(q,a) = H_0^{(1)}(q,a)\cdot \begin{cases}
        \frac{1}{14} & \text{ if } \nu_2(q) = 0,\\
        \frac{1}{2(7e_2(a)-1)} & \text{ if }\nu_2(q) = 1,\\
        0 & \text{ if } \nu_2(q) \ge 2.
    \end{cases}
\end{equation}
Note in particular, for the $\nu_2(q) = 1$ case, if we write $q = 2m$ and use the second equations from both \Cref{prop:main terms for S,prop:multiplicative sums}, we get
\begin{equation*}
    H_0^{(2)}(q,a) = e_2(a)\frac{C(m)}{14m}\prod_{p:\nu_p(m) \ge 2}\Theta_{\nu_p(m)}(p)\prod_{p:\nu_p(m) = 1}\qty(\frac{3p^2-3p+1}{(p-1)^2}-2\frac{p-1}{p}\cos\qty(\frac{2\pi a\overline{(q/p)}^{(p)}}{p})).
\end{equation*}
So we recover the factor $\frac{1}{2(7e_2(a)-1)}$ after observing that
    \begin{equation*}
        \begin{aligned}
            \eval{\qty(\frac{3p^2-3p+1}{(p-1)^2}-2\frac{p-1}{p}\cos\qty(\frac{2\pi a\overline{(q/p)}^{(p)}}{p}))}_{p=2}& = 7-e_2(a),
        \end{aligned}
    \end{equation*}
    and that
    \begin{equation*}
         C(q) = C(m)\cdot\qty(\eval{1+\frac{5p^2-6p+2}{p^2(p-1)^2}}_{p=2})^{-1} = \frac{2}{7}C(m).
    \end{equation*}
    Since we want to work with $a/q$, we can assume that $(a,q) = 1$ (or $a=0,q=1$) to get
    \begin{equation*}
        H_0(q,a) = H_0^{(1)}(q,a)\cdot\begin{cases}
            \frac{15}{14} & \text{if } \nu_2(q) = 0,\\
            \frac{15}{16} & \text{if } \nu_2(q) = 1, \\
            1 & \text{if } \nu_2(q) \ge 2.
        \end{cases}
    \end{equation*}
    Recalling the definitions of our local factors from \Cref{prop:main terms for S,prop:multiplicative sums}, $C(q)\geq (\log 2)^3$, $\Theta_\ell(p) \geq \binom{\ell + 2}{2}\ge 3$ for all $\ell\ge 1$, and
    \begin{equation*}
        \frac{3p^2-3p+1}{(p-1)^2}-2\frac{p-1}{p}\cos\qty(\frac{2\pi a\overline{(q/p)}^{(p)}}{p})\ge  \frac{3p^2-3p+1}{(p-1)^2}-2\frac{p-1}{p}\ge 1.
    \end{equation*}
    So we deduce that for all $q\in \mathbb N$, $(a,q) = 1$ (or $q = 1,\ a = 0$),
    \begin{equation*}
        H_0(q,a)\ge \frac{15}{16q}(\log 2)^3.
    \end{equation*}
    For the upper bound, we have:
    \begin{equation*}
         \frac{3p^2-3p+1}{(p-1)^2}-2\frac{p-1}{p}\cos\qty(\frac{2\pi a\overline{(q/p)}^{(p)}}{p}) \le \frac{3p^2-3p+1}{(p-1)^2} + 2\frac{p-1}{p}\le 8,
    \end{equation*}
    and using $\Theta_\ell(p)\le \ell^2$ and $C(q) \le C(1)$, we conclude by using the divisor bound again:
    \begin{equation*}
        H_0(q,a)\le \frac {C(1)}q \qty(\prod_{p:\nu_p(q)\ge 2}\nu_p(q))^2\prod_{p:\nu_p(q) = 1} 8\ll \frac1q\tau(q)^28^{\omega(q)} \ll \frac1{q^{1-o(1)}}. \qedhere
    \end{equation*}
\end{proof}

\appendix

\section{Standard results and approximations for arithmetic functions}

\begin{prop}\label{reality of integral}
For any finite union of intervals $I\subseteq S^1 = [0,1]$, which has $I = -I$, and any function $F:\mathbb Z\to \mathbb R$,
\begin{equation*}
    \int_I \widehat{F}(\alpha)\dd\alpha\in\mathbb R.
\end{equation*}
\end{prop}

\begin{proof}
First observe that $\widehat{F}(-\alpha) = \overline{\widehat{F}(\alpha)}$ and that we can substitute $\beta = -\alpha$ into the integral above.

Then by assumption, the range of integration merely switches sign, which is cancelled by the Jacobian factor, also $-1$.
\end{proof}

\begin{prop}[Rational approximations for minor arcs]\label{rational approximations for minor arcs}
For $\alpha \in \mathfrak m(R)$, get a rational $a/q$ with $R<q\le N/R$ (and $a,q$ coprime) such that
\begin{equation*}
    \qty|\alpha-\frac{a}{q}|\leq q^{-2}.
\end{equation*}
\end{prop}
\begin{proof}
Use Dirichlet's approximation theorem to extract $a,q$ with $q\leq N/R$, and
\begin{equation*}
    \qty|\alpha-\frac aq | \leq \frac{1}{q(N/R)}\leq R/N.
\end{equation*}
If $q\leq R$ then by definition, $\alpha \in \mathfrak M(R)$, a contradiction, so we are done.
\end{proof}

\begin{prop}\label{psi and theta bound}
    Let $\psi$ and $\vartheta$ be the Chebyshev functions, i.e.
    \begin{equation*}
        \psi(X) = \sum_{n\le X}\Lambda(n),\quad \vartheta(X) = \sum_{p\leq X}\log p.
    \end{equation*}
    Then
    \begin{equation*}
        \abs{\psi(X)-\vartheta(X)}\ll \sqrt X.
    \end{equation*}
\end{prop}
\begin{proof}
    We see that the difference $\psi(X) - \vartheta(X)$ simply sums $\log p $ for all $p\leq X$ with $p^k\leq X$ for some $k\geq 2$. This exponent is then at most $\frac{\log X}{\log p}$, and $p\le \sqrt X$. So recalling Chebyshev's theorem $(\pi(X)\asymp X/\log X)$,
    \begin{equation*}
        \abs{\psi(X)-\vartheta(X)}\leq \sum_{p\leq \sqrt{X}}\qty(\frac{\log X}{\log p})\log p \le \log X\cdot\pi\qty(\sqrt X) \ll \sqrt X.
    \end{equation*}
\end{proof}

\begin{lem}[Alternative form for Ramanujan sums]\label{alternative form for Ramanujan sums} 
For $r\in\mathbb N$ and $n\in\mathbb Z$,
\begin{equation*}
    c_r(n)=\sum_{d\mid (r,n)}d\mu\left(\frac{r}{d}\right).
\end{equation*}
\end{lem}
\begin{proof}
    By M\"obius inversion, $\mathbf 1_{(\ell ,r)=1}=\sum_{d\mid (\ell,r)}\mu(d)$, so we get
    \[c_r(n) = \sum_{\ell\in(\mathbb Z/r\mathbb Z)^\ast}e\qty(\frac{n\ell}{r})=\sum_{\ell(r)}\mathbf 1_{(\ell,r)=1}e\qty(\frac{n\ell}{r})=\sum_{\ell(r)}\sum_{d\mid (\ell,r)}\mu(d)e\qty(\frac{n\ell}{r})=\sum_{d\mid r}\mu(d)\sum_{\substack{\ell(r)\\ d\mid \ell}}e\qty(\frac{n\ell}{r}).\]
    Write $\ell=d\ell'$ with $\ell'$ taken modulo $r/d$. Then
    \[\sum_{\substack{\ell(r)\\ d\mid \ell}}e\qty(\frac{n\ell}{r})=\sum_{\ell'(r/d)}e\qty(\frac{nd\ell'}{r})=\sum_{\ell'(r/d)}e\qty(\frac{n\ell'}{r/d}).\]
    The inner sum is a complete geometric sum over the $r/d$-th roots of unity. It is $0$ unless $r/d\mid n$, in which case it is $r/d$. Thus,
    \begin{align*}
        c_r(n)=\sum_{\substack{d\mid r\\ r/d\mid n}}\mu(d)\frac{r}{d}=\sum_{\substack{d'\mid r\\ d'\mid b}}d'\mu\qty(\frac{r}{d'}).&\qedhere
    \end{align*}
\end{proof}

\begin{lem}\label{sum over min}
Let $|\alpha-a/q|\le q^{-2}$, $(a,q)=1$, and $X,q\ge 1,x>1.$ We have
 \[\sum_{l\le X}\min\left\{\frac{x}l,\frac{1}{\|\alpha l\|_{\mathbb Z}}\right\}\ll \qty(\frac{x}{q}+q)\log (2Xq).\]
\end{lem}
\begin{proof}
    Write $\alpha=a/q+\beta$ with $|\beta|\le q^{-2}$ and $l=hq+r$ with $1\le r\le q.$ Then,
    \begin{align*}
        \sum_{l\le X}\min\left\{\frac{x}{l},\frac{1}{\|\alpha l\|_{\mathbb Z}}\right\}\le\sum_{0\le h\le X/q}\sum_{1\le r\le q}\min\left\{\frac{x}{hq+r},\frac{1}{\left\|\frac{ra}{q}+hq\beta+r\beta\right\|_{\mathbb Z}}\right\}.
    \end{align*}
For $1\le r\le q/2,$ we have $|r\beta|\le 1/(2q).$ Thus, the contribution from $h=0$ and $r\le q/2$ is bounded by
\begin{align*}
    \ll \sum_{r\le q/2}\left(\left\|\frac{ra}{q}\right\|-\frac{1}{2q}\right)^{-1}\ll \sum_{r\le q}\frac{q}{r}\ll q\log q.
\end{align*}
For the other terms, we have $hq+r\gg(h+1)q$, so that their contribution is bounded by
\begin{align*}
    \sum_{0\le h\le X/q}\sum_{1\le r\le q}\frac{x}{(h+1)q} \ll \frac xq \log X.
\end{align*}
\end{proof}

\begin{lem}\label{A.4 in Green}
    For $\beta\in\mathbb R$,
    \begin{equation*}
        \sum_{n=1}^N\qty(1-\frac{n}{N})\cos(2\pi\beta n)=\frac{1}{2N}\qty(\frac{\sin(\pi\beta N)}{\sin(\pi\beta)})^2-\frac{1}{2}
    \end{equation*}
    and
    \begin{equation*}
        \qty|\sum_{n=1}^N\qty(1-\frac{n}{N})e(\beta n)|\ll\min\left\{N,\frac{1}{\|\beta\|_{\mathbb Z}}\right\}.
    \end{equation*}
\end{lem}

\begin{proof}
First, consider
\begin{align*}
    \sum_{n=-N}^N\qty(1-\frac{|n|}{N})e(\beta n)&=\frac{1}{N}\sum_{0\le m,n<N}e(\beta(m-n))=\frac{1}{N}\qty|\sum_{0\le n\le N-1}e(n\beta)|^2\\
    &=\frac{1}{N}\qty|\frac{e(\beta N)-1}{e(\beta)-1}|=\frac{1}{N}\qty(\frac{\sin(\pi\beta N)}{\sin(\pi\beta)})^2.
\end{align*}
Subtracting 1 (corresponding to the $n=0$ term) and using symmetry, we get the first identity. 

For the second equation, we can differentiate the geometric sum formula to get
\[\sum_{n=1}^N\qty(1-\frac{n}{N})z^n=\frac{-z^{N+1}/N-z}{z-1}-\frac{z^{N+2}+z}{N(z-1)^2}.\]
Setting $z=e(\beta)$ and using $|e(\beta)-1|\gg\norm{\beta}_\mathbb Z$, we get
\[\qty|\sum_{n=1}^N\qty(1-\frac{n}{N})e(\beta n)|\ll\frac{1}{\norm{\beta}_\mathbb Z}+\frac{1}{N\norm{\beta}_\mathbb Z^2}.\]
Together with the trivial bound of $\ll N$, we get the desired inequality by considering the cases $\norm{\beta}_\mathbb Z\le 1/N$ and $\norm{\beta}_\mathbb Z\ge 1/N.$
\end{proof}

\begin{lem}\label{probability lemma}
Let $\xi>1$, $\ell\in\mathbb Z_{\ge 0}$. Then,
\begin{equation*}
    \sum_{\substack{e_{1},e_{2},e_{3}\ge 0\\ e_{1}+e_{2}+e_{3}\ge \ell}}\frac{1}{\xi^{e_{1}+e_{2}+e_{3}}}=\qty(1-\frac{1}{\xi})^{-3}\frac{1}{\xi^{\ell+2}}\qty[\binom{\ell+2}{2}(\xi-1)^2+(\ell+2)(\xi-1)+1].
\end{equation*}
\end{lem}
\begin{proof} Let $X_1,X_2,X_3$ be identical and independent random variables on $\mathbb Z_{\ge 0}$, with 
\[\mathbb P(X_1=j)=\frac{1}{\xi^j}\qty(1-\frac{1}{\xi}),\]
i.e. a shifted geometric distribution. Then $X_1+X_2+X_3$ is the number of failures before the third success in a Bernoulli trial, of parameter $1-1/\xi$. So $X_1+X_2+X_3\ge \ell$ means that there are at least $\ell$ failures before the third success, or, equivalently, that there are fewer than 3 successes among the first $\ell+2$ trials:
\[\qty(1-\frac{1}{\xi})^3\sum_{\substack{e_{1},e_{2},e_{3}\ge 0\\ e_{1}+e_{2}+e_{3}\ge \ell}}\frac{1}{\xi^{e_{1}+e_{2}+e_{3}}}=\mathbb P(X_1+X_2+X_3\ge \ell)=\sum_{j=0}^2\binom{\ell+2}{j}\qty(1-\frac{1}{\xi})^j\frac{1}{\xi^{\ell+2-j}},\]
as claimed.
\end{proof}

\begin{lem}\label{lem:ben green b.2}
Given a constant $B\geq 0$ and $X\ge 1$, have
\begin{equation*}
    \sum_{n\leq X}\mu(n)^2\frac{\tau(n)^B}{n}\ll \log^{\order{1}}X.
\end{equation*}
\end{lem}
\begin{proof}
    This follows by setting $h = 1$ in lemma B.2, \cite{green2023sarkozystheoremshiftedprimes}.
\end{proof}

\begin{lem}\label{lem: character sum}
    For $X\ge 0$, let
\begin{equation*}
\Psi(X,r,n):=\sum_{\ell\in(\mathbb Z/r\mathbb Z)^*}e_r(\ell n)\sum_{m\le X}\Lambda(m)e_r(-\ell m).
\end{equation*} 
Then, assuming GRH,
\begin{equation*}
    \Psi(X,r,n)=\frac{\mu(r)}{\phi(r)}c_r(n)X+\order{r\sqrt{X}(\log(rX))^2}.
\end{equation*}
\end{lem}

\begin{proof}
    For a Dirichlet character $\chi \bmod r$, we write
\[\psi(X,\chi):=\sum_{n\le X}\Lambda(n)\chi(n).\]
Observe that $\chi(n) = 0$ for $(n,r) > 1$, so the sum is in fact over $n$ for which $(n,r) = 1$.

Now we use the following (see, for example, \cite{Granville2006Large}): if $\chi \bmod r$ is a non-trivial
Dirichlet character, then on GRH, we have
\[|\psi(X,\chi)| \ll \sqrt{X}(\log (rX))^2.\]
If $\chi_0 \bmod r$ is the trivial character, then we have (on RH)
\[\psi(X,\chi_0)=X+\order{\sqrt{X}(\log( rX))^2}.\]
Using the orthogonality of Dirichlet characters, for $(\ell m,r)=1$, we write
\[e_r(\ell m)=\frac{1}{\phi(r)}\sum_{\chi\bmod r}\tau(\overline{\chi})\chi(\ell)\chi(m),\]
where
\[\tau(\chi):=\sum_{t\bmod r}\chi(t)e_r(t)\]
is the Gauss sum. Thus, if $(\ell,r)=1$, we can split according to whether $(m,r)=1$ to get
\[\sum_{m\le X}\Lambda(m)e_r(\ell m)=\frac{1}{\phi(r)}\sum_{\chi\bmod r}\tau(\overline{\chi})\chi(\ell)\psi(X,\chi)+\order{\log(rX)^2}.\]
The error term arises when $(m,r)>1$ and $\Lambda(m)\ne0$, i.e., when $m$ is a power of a prime dividing $r$; there are $\order{\log r}$ such $m\le X$, each contributing $\order{\log X}$, for a total of $\order{\log r\log X}$. 
Now 
\begin{align*}
\Psi(X,r,n)&=\frac{1}{\phi(r)}\sum_{\chi\bmod r}\tau(\overline{\chi})\psi(X,\chi)\sum_{\ell \in (\mathbb Z/r\mathbb Z)^\ast} e_r(\ell n)\chi(-\ell)+\order{\phi(r)(\log(rX))^2}\\
&=\frac{1}{\phi(r)}\sum_{\chi\bmod r}\chi(-1)\tau(\overline{\chi})\psi(X,\chi)\sum_{\ell \in (\mathbb Z/r\mathbb Z)^\ast} \chi(\ell)e_r(\ell n)+\order{\phi(r)(\log(rX))^2}.
\end{align*}
At the principal term $\chi=\chi_0$, $\chi(-1)=1$, $\tau(\chi_0)=c_r(1)=\mu(r)$ and 
\(\sum_\ell\chi(\ell)e_r(\ell n)=c_r(n)\), so, on RH, the contribution at $\chi_0$ is
\[\frac{\mu(r)}{\phi(r)}c_r(n)\psi(X,\chi_0)=\frac{\mu(r)}{\phi(r)}c_r(n)\qty(X+\order{\sqrt{X}(\log(rX))^2})=\frac{\mu(r)}{\phi(r)}c_r(n)X+\order{\sqrt{X}(\log(rX))^2}\]
where we have used the trivial bound $|c_r(n)|\le\phi(r).$ At the non-principal terms, we use Cauchy-Schwarz to get
$$
\qty|\sum_{\chi\ne\chi_0}\chi(-1)\tau(\overline\chi)\psi(X,\chi)\sum_{\ell} \chi(\ell)e_r(\ell n)|
\le \max_{\chi\ne\chi_0}|\psi(X,\chi)|\sqrt{\sum_{\chi}|\tau(\overline\chi)|^2}\sqrt{\sum_{\chi}\qty|\sum_{\ell}\chi(\ell)e_r(\ell n)|^2}.
$$
But on GRH, $\max_{\chi\ne\chi_0}|\psi(X,\chi)|=\order{\sqrt X(\log(rX))^2}$, and by orthogonality of Dirichlet characters mod $r$, we have
\[\sum_{\chi}|\tau(\overline\chi)|^2=\sum_{\chi}\qty|\sum_\ell\chi(\ell)e_r(\ell n)|^2=\phi(r)^2.\]
Hence the non-principal contribution is
\[\frac{1}{\phi(r)}\order{\sqrt X(\log(rX))^2}\phi(r)^2=\order{\phi(r)\sqrt{X}(\log(rX))^2}.\]
As $\phi(r)\le r$, we get
\begin{align*}
    \Psi(X,r,n)=\frac{\mu(r)}{\phi(r)}c_r(n)X+\order{r\sqrt{X}(\log(rX))^2}.&\qedhere
\end{align*}
\end{proof}

\raggedright
\printbibliography[title={References}]

@misc{green2023sarkozystheoremshiftedprimes,
      title={On S\'ark\"ozy's theorem for shifted primes}, 
      author={Ben Green},
      year={2023},
      eprint={2206.08001},
      archivePrefix={arXiv},
      primaryClass={math.NT},
      url={https://arxiv.org/abs/2206.08001}, 
}

@article{Granville2006Large,
	author = {Granville, Andrew and Soundararajan, K.},
	journal = {Journal of the American Mathematical Society},
	number = {2},
	year = {2006},
	month = {may 26},
	pages = {357--384},
	publisher = {American Mathematical Society (AMS)},
	title = {Large character sums: Pretentious characters and the {P}{\' o}lya-{Vinogradov} theorem},
	volume = {20},
}

@misc{greenunpublished,
      title={An improved bound for S\'ark\"ozy’s theorem for shifted primes, assuming GRH}, 
      author={Ben Green},
      year = {manuscript available on request.}, 
}

@article{Sarkozy1978On,
	author = {S{\' a}rk{\" o}zy, A.},
	journal = {Acta Mathematica Academiae Scientiarum Hungaricae},
	number = {1-2},
	year = {1978},
	month = {3},
	pages = {125--149},
	publisher = {{Springer Science and Business Media LLC}},
	title = {On difference sets of sequences of integers. {I}},
	volume = {31},
}

@article{Sitaramachandrarao1985On,
	author = {Sitaramachandrarao, R.},
	journal = {Rocky Mountain Journal of Mathematics},
	number = {2},
	year = {1985},
	month = {6},
	publisher = {Rocky Mountain Mathematics Consortium},
	title = {On an error term of {Landau}-{II}},
	volume = {15},
}

@book{Koukoulopoulos2019distribution,
	author = {Koukoulopoulos, Dimitris},
	year = {2019},
	publisher = {AMS},
	title = {The distribution of prime numbers},
}

@article{LeSpencer,
    author = {L\^e, T.H. and Spencer, C. V.},
    title = {Difference sets and the irreducibles in function fields},
    journal = {Bulletin of the London Mathematical Society},
    year = 2011
}

@article{ThornerZaman,
    author = {Thorner, J. and Zaman, A.},
    title = {An explicit version of Bombieri's log-free density estimate and S\'ark\"ozy's theorem for shifted primes},
    journal = {Forum Math.},
    year = 2024
}

@misc{FanLott,
    title = {S\'ark\"ozy’s theorem in $\mathbb F_q[t]$ via the van der Corput property},
    author = {Fan, S. and Lott, A.},
      year={2026},
      archivePrefix={arXiv},
      primaryClass={math.NT},
      url={https://arxiv.org/abs/2510.27581}, 
}

@article{Lucier2008Difference,
	author = {Lucier, J.},
	journal = {Acta Mathematica Hungarica},
	number = {1-2},
	year = {2008},
	month = {jan 19},
	pages = {79--102},
	publisher = {{Springer Science and Business Media LLC}},
	title = {Difference sets and shifted primes},
	volume = {120},
}

@article{Ruzsa2008Difference,
	author = {Ruzsa, Imre Z. and Sanders, Tom},
	journal = {Acta Arithmetica},
	number = {3},
	year = {2008},
	pages = {281--301},
	publisher = {Institute of Mathematics, Polish Academy of Sciences},
	title = {Difference sets and the primes},
	volume = {131},
}

@article{Wang2020On,
	author = {Wang, Ruoyi},
	journal = {Journal of Number Theory},
	year = {2020},
	month = {6},
	pages = {220--234},
	publisher = {Elsevier BV},
	title = {On a theorem of {S}{\' a}rk{\" o}zy for difference sets and shifted primes},
	volume = {211},
}

@article{Heath1985ternary,
	author = {Heath-Brown, David Rodney},
	journal = {Revista Matem{\' a}tica Iberoamericana},
	number = {1},
	year = {1985},
	month = {mar 31},
	pages = {45--59},
	publisher = {European Mathematical Society - EMS - Publishing House GmbH},
	title = {The ternary goldbach problem},
	volume = {1},
}

@misc{GrimmeltTeräväinen,
      title={The Exceptional Set in Goldbach's Problem with two Chen Primes}, 
      author={Grimmelt, L. and Ter\"av\"ainen, J.},
      year={2026},
      archivePrefix={arXiv},
      primaryClass={math.NT},
      url={https://arxiv.org/abs/2508.16400}, 
}

@article{Roth1953On,
	author = {Roth, K. F.},
	journal = {Journal of the London Mathematical Society},
	number = {1},
	year = {1953},
	month = {1},
	pages = {104--109},
	publisher = {Wiley},
	title = {On {Certain} {Sets} of {Integers}},
	volume = {s1-28},
}

\end{document}